\documentclass[a4paper,12pt]{amsart}

\usepackage[top=3cm,left=2.5cm,right=2.5cm,bottom=3cm]{geometry}
\usepackage{relsize}
\usepackage{lineno}
\usepackage{hyperref}
\usepackage{amsmath,amsthm,pb-diagram,amssymb,comment}
\usepackage{amsfonts,graphicx,color}
\usepackage{enumitem, fancyhdr, dsfont}
\usepackage{thmtools,cleveref}
\usepackage[normalem]{ulem}
\usepackage{stmaryrd}
\usepackage{textcomp}
\usepackage{multicol}
\usepackage{contour}
\usepackage{mathbbol}
\usepackage{fontawesome}
\usepackage{tikz,float}
\usepackage{chemformula}
\usepackage{amsthm,thmtools,xcolor} 

\allowdisplaybreaks

\hypersetup{
  colorlinks=true,
  linkcolor=capri,
  filecolor=capri,
  urlcolor=capri,
  citecolor=capri,
}

\newcommand{\dom}{\mathrm{dom}}

\newcommand{\thzfc}{\mathrm{ZFC}}
\newcommand{\Con}{\mathrm{Con}}

 \newcommand{\Ed}{\mathbf{Ed}}
\newcommand{\Mbf}{\mathbf{M}}

\newcommand{\Bwf}{\mathcal{B}}

\newcommand{\Ewf}{\mathcal{E}}

\newcommand{\Iwf}{\mathcal{I}}
\newcommand{\Jwf}{\mathcal{J}}

\newcommand{\Mwf}{\mathcal{M}}
\newcommand{\Nwf}{\mathcal{N}}

\newcommand{\bfrak}{\mathfrak{b}}
\newcommand{\cfrak}{\mathfrak{c}}
\newcommand{\dfrak}{\mathfrak{d}}
\newcommand{\efrak}{\mathfrak{e}}

\newcommand{\menos}{\smallsetminus}

\newcommand{\frestr}{{\restriction}}

\newcommand{\add}{\mbox{\rm add}}
\newcommand{\cov}{\mbox{\rm cov}}

\newcommand{\non}{\mbox{\rm non}}
\newcommand{\cof}{\mbox{\rm cof}}
\newcommand{\limdir}{\mbox{\rm limdir}}

\newcommand{\Bor}{\mathds{B}}
\newcommand{\Cor}{\mathds{C}}

\newcommand{\Por}{\mathds{P}}
\newcommand{\Qor}{\mathds{Q}}

\newcommand{\Qnm}{\dot{\Qor}}

\newcommand{\R}{\mathbb{R}}

\newcommand{\cf}{\mbox{\rm cf}}

\newcommand{\la}{\langle}
\newcommand{\ra}{\rangle}

\newcommand{\Seq}{\mathrm{seq}}
\newcommand{\Fr}{\mathrm{Fr}}
\newcommand{\Rbf}{\mathbf{R}}
\newcommand{\Ebf}{\mathbf{E}}

\newcommand{\Cbf}{\mathbf{C}}
\newcommand{\Cv}{\mathrm{Cv}}

\newcommand{\aLc}{\mathbf{aLc}}
\newcommand{\Lb}{\mathbf{Lb}}

\newcommand{\id}{\mathrm{id}}

\newcommand{\balc}{\mathfrak{b}^{\mathrm{aLc}}}
\newcommand{\dalc}{\mathfrak{d}^{\mathrm{aLc}}}

\newcommand{\vfa}{\mathfrak{v}}

\newcommand{\leqT}{\preceq_{\mathrm{T}}}
\newcommand{\eqT}{\cong_{\mathrm{T}}}

\newcommand{\ubd}{\mathbf{ubd}}
\newcommand{\hyp}{\mathbf{hyp}}
\newcommand{\gen}{\mathrm{gen}}

\newcommand{\const}{\mathsf{const}}

\newcommand{\spr}{\mathrm{pr}}
\newcommand{\cpr}{\mathrm{cp}}

\newcommand{\compact}{\mathrm{K}}

\newcommand{\setsep}{:\,}
\newcommand{\set}[2]{\{#1\setsep#2\}}
\newcommand{\largeset}[2]{\left\{#1\setsep#2\right\}}
\newcommand{\seq}[2]{\la#1\setsep#2\ra}
\newcommand{\bsp}{\allowbreak\ }
\newcommand{\comma}{\mathord,\bsp}
\newcommand{\setand}{\bsp\text{and}\bsp}

\newcommand{\baire}[1][\omega]{{}^\omega#1}

\newcommand{\cantor}{\baire[2]}

\newcommand{\xprod}{\prod\nolimits^*}
\newcommand{\predby}{\sqsubset}
\newcommand{\rel}[1][r]{\mathrel{\mathbf{#1}}}
\newcommand{\vfrak}{\mathfrak{v}}

\contourlength{0.8pt}

\contourlength{0.8pt}

	\definecolor{ultramarineblue}{rgb}{0.25, 0.4, 0.96}
\definecolor{cornellred}{rgb}{0.7, 0.11, 0.11}
\definecolor{cobalt}{rgb}{0.0, 0.28, 0.67}
\definecolor{bleudefrance}{rgb}{0.19, 0.55, 0.91}
\definecolor{darkblue}{rgb}{0.0, 0.0, 0.55}
\definecolor{ferrarired}{rgb}{1.0, 0.11, 0.0}
\definecolor{brandeisblue}{rgb}{0.0, 0.44, 1.0}
\definecolor{azure(colorwheel)}{rgb}{0.0, 0.5, 1.0}
\definecolor{aqua}{rgb}{0.0, 1.0, 1.0}
\definecolor{aguamarina}{cmyk}{0.85,0,0.33,0}
\definecolor{cafe}{cmyk}{0,0.81,1,0.60}
\definecolor{capri}{rgb}{0.0, 0.75, 1.0}
\definecolor{canela}{cmyk}{0.14,0.42,0.56,0}
\definecolor{darkgray}{cmyk}{0,0,0,0.50}
\definecolor{emerald}{cmyk}{0.91,0,0.88,0.12}
\definecolor{fresa}{cmyk}{0,1,0.50,0}
\definecolor{gold}{cmyk}{0,0.10,0.84,0}
\definecolor{lightgray}{cmyk}{0,0,0,0.30}
\definecolor{marron}{cmyk}{0,0.72,1,0.45}
\definecolor{melon}{cmyk}{0,0.29,0.84,0}
\definecolor{ladri}{cmyk}{0,0.77,0.87,0}
\definecolor{olive}{cmyk}{0.64,0,0.95,0.40}
\definecolor{orange}{cmyk}{0,0.42,1,0}
\definecolor{peach}{cmyk}{0,0.46,0.50,0}
\definecolor{pink}{cmyk}{0,0.10,0.10,0}
\definecolor{orange}{cmyk}{0,0.42,1,0}
\definecolor{pine}{cmyk}{0.92,0,0.59,0.25}
\definecolor{purple}{cmyk}{0.45,0.86,0,0}
\definecolor{violet}{cmyk}{0.07,0.90,0,0.34}
\definecolor{craneorange}{RGB}{252,187,6}
\definecolor{red(ncs)}{rgb}{0.77, 0.01, 0.2}

\definecolor{aguamarina}{cmyk}{0.85,0,0.33,0}
\definecolor{cafe}{cmyk}{0,0.81,1,0.60}
\definecolor{canela}{cmyk}{0.14,0.42,0.56,0}
\definecolor{darkgray}{cmyk}{0,0,0,0.50}
\definecolor{emerald}{cmyk}{0.91,0,0.88,0.12}
\definecolor{fresa}{cmyk}{0,1,0.50,0}
\definecolor{gold}{cmyk}{0,0.10,0.84,0}
\definecolor{lightgray}{cmyk}{0,0,0,0.30}
\definecolor{marron}{cmyk}{0,0.72,1,0.45}
\definecolor{melon}{cmyk}{0,0.29,0.84,0}
\definecolor{ladri}{cmyk}{0,0.77,0.87,0}
\definecolor{olive}{cmyk}{0.64,0,0.95,0.40}
\definecolor{orange}{cmyk}{0,0.42,1,0}
\definecolor{peach}{cmyk}{0,0.46,0.50,0}
\definecolor{pink}{cmyk}{0,0.10,0.10,0}
\definecolor{orange}{cmyk}{0,0.42,1,0}
\definecolor{pine}{cmyk}{0.92,0,0.59,0.25}
\definecolor{purple}{cmyk}{0.45,0.86,0,0}
\definecolor{violet}{cmyk}{0.07,0.90,0,0.34}

\newcommand{\lqq}{\textup{``}}
\newcommand{\rqq}{\textup{''}}

\DeclareSymbolFont{extraup}{U}{zavm}{m}{n}
\DeclareMathSymbol{\varheart}{\mathalpha}{extraup}{86}
\DeclareMathSymbol{\vardiamond}{\mathalpha}{extraup}{87}

\definecolor{dodger}{rgb}{0.0,0.5,1.0}

\definecolor{amber}{rgb}{1.0,0.49,0.0}

\definecolor{ogreen}{RGB}{107,142,35}

\title[Constant prediction and evasion number]{Constant prediction and evasion number, I: Generalization and variants}

\author{Miguel A. Cardona}
\address[Miguel A. Cardona]{Einstein Institute of Mathematics\\
Edmond J. Safra Campus, Givat Ram\\
The Hebrew University of Jerusalem\\
Jerusalem, 91904, Israel}
\email{\href{mailto:miguel.cardona@mail.huji.ac.il}{miguel.cardona@mail.huji.ac.il}}
\urladdr{\url{https://sites.google.com/view/miacardonamo}}

\author{Miroslav Repick\'y}
\address[Miroslav Repick\'y]{Mathematical Institute\\
Slovak Academy of Sciences\\
Gre\v{s}\'akova~6\\
040\,01 Ko\v{s}ice\\
Slovak Republic}
\email{\href{mailto:repicky@saske.sk}{repicky@saske.sk}}

\thanks{The first would like to thank the Israel Science Foundation for partially supporting this research by grant 2320/23 (2023-2027); and the second author was supported by the grant VEGA 2/0104/24 of the Slovak Grant Agency VEGA}



\definecolor{sub0}{RGB}{29,32,137}
\definecolor{sub1}{RGB}{1,71,157}
\definecolor{sub2}{RGB}{1,104,183}
\definecolor{sub3}{RGB}{0,160,234}
\definecolor{sug}{RGB}{0,154,68}
\definecolor{suy}{RGB}{208,219,1}

\DeclareUnicodeCharacter{0301}{\'{e}}

\begin{document}

\makeatletter
\def\@roman#1{\romannumeral #1}
\newcommand{\startlist}{\ \@beginparpenalty=10000}
\makeatother

\newcounter{enuAlph}
\renewcommand{\theenuAlph}{\Alph{enuAlph}}

\numberwithin{equation}{section}


\theoremstyle{plain}
\newtheorem{theorem}{Theorem}[section]
\newtheorem{corollary}[theorem]{Corollary}
\newtheorem{lemma}[theorem]{Lemma}
\newtheorem{mainlemma}[theorem]{Main Lemma}
\newtheorem{proposition}[theorem]{Proposition}
\newtheorem{clm}[theorem]{Claim}
\newtheorem{fact}[theorem]{Fact}
\newtheorem{exer}[theorem]{Exercise}
\newtheorem{question}[theorem]{Question}
\newtheorem{problem}[theorem]{Problem}
\newtheorem{subclm}[theorem]{Subclaim}
\newtheorem{conjecture}[theorem]{Conjecture}
\newtheorem{assumption}[theorem]{Assumption}
\newtheorem{hopeth}[theorem]{Hopeful Theorem}
\newtheorem{hopele}[theorem]{Hopeful Lemma}
\newtheorem{discussion}[theorem]{Discussion}
\newtheorem{remark}[theorem]{Remark}
\newtheorem{challenging}[enuAlph]{Main challenging}
\newtheorem{teorema}[enuAlph]{Theorem}
\newtheorem*{mainthm}{Main Theorem}
\newtheorem*{thm}{Theorem}
\newtheorem*{corolario}{Corollary}

\theoremstyle{definition}
\newtheorem{definition}[theorem]{Definition}
\newtheorem{example}[theorem]{Example}
\newtheorem{notation}[theorem]{Notation}
\newtheorem{context}[theorem]{Context}
\newtheorem*{defi}{Definition}
\newtheorem*{acknowledgements}{Acknowledgements}

\def\sectionautorefname{Section}
\def\subsectionautorefname{Subsection}


\begin{abstract}
Using the concept of constant evasion to different sorts of suitable binary relations,  we establish many cardinal invariants derived from the established cardinal invariants $\mathfrak{e}^\mathrm{const}_{n}$ and $\mathfrak{v}^\mathrm{const}_{n}$, called the constant evasion number and the constant prediction number. We formulate several limits and consistency results pertaining to them. 
\end{abstract}
\maketitle


\makeatother

\section{Introduction}\label{intro}

This research forms part of the study of evading and predicting theory, a combinatorial framework originally introduced by Blass~\cite{blasse}. Since its introduction, this area has been extensively developed by several authors, including ~\cite{BrendlevasionI, BreIII,BreShevaIV,kamoeva,kamopred,kadinfgam,Kadagenpredi,BreGarV,CRS}. These works have deepened the understanding of the combinatorial and set-theoretic aspects of evading and predicting, establishing numerous connections with other cardinal invariants, e.g., those in Cichon's diagram and forcing constructions.

A refinement of these notions, termed constant evasion and constant prediction, was introduced by Kamo~\cite{kamoeva}. Here, the requirements for successful evasion or prediction are strengthened: instead of asymptotic correctness, one demands that within every bounded interval of prescribed length, there exists a correct prediction or evasion. This strengthening yields a sharper combinatorial distinction and gives rise to two cardinal invariants: $\efrak^\const_{n}$ and $\vfrak^\const_{n}$, the constant evasion and the constant prediction numbers, respectively, are presented below.

Let $n\leq\omega$ and call a function $\sigma\colon{}^{<\omega}n\to n$ a predictor. By $\Sigma_n$, we denote the set of all predictors. Say \emph{$\sigma$ constantly predicts a real $f\in{}^{\omega}n$}, written as $f\predby^\cpr_{=}\sigma$ 
\begin{align*}
f\predby^\cpr_{=}\sigma&\text{ iff }
\exists k\in\omega\ \forall i\in\omega\ \exists j\in[i,i+k)\
f(j)=\sigma(f\frestr j).
\end{align*}
We define the constant evasion number
\[\efrak^\const_{n}:=\min\set{|F|}{F\subseteq{}^{\omega}n\setand\neg\exists\sigma\in\Sigma_n\ \forall f\in F\colon f\predby^\cpr_{=}\sigma}\]
and the constant prediction number
\[\vfrak^\const_{n}:=\min\set{|S|}{S\subseteq\Sigma_n\setand\forall f\in{}^{\omega}n\ \exists\sigma\in S\colon f\predby^\cpr_{=}\sigma}.\]

Inspired by the previous cardinal invariants, this paper aims to generalize and establish different versions of these cardinals. Afterwards, we derived some bounds and consistency results for them. Furthermore, using these new cardinals, we provide some characterization cardinal invariants from Cicho\'n's diagram.

In addition to the already mentioned cardinals $\efrak_n^\const$ and $\vfa_n^\const$, we shall refer to a few additional well-known cardinal invariants below. 

Given a formula $\phi$, $\forall^\infty n<\omega\colon \phi$ means that all but finitely many natural numbers satisfy~$\phi$; $\exists^\infty n<\omega\colon \phi$ means that infinitely many natural numbers satisfy $\phi$. For $f,g\in\baire$ define
\[f\leq^*g\text{ iff } \forall^\infty n\in\omega\colon f(n)\leq g(n).\]
We recall
\begin{align*}
\bfrak &:=\min\set{|F|}{F\subseteq\baire\setand\forall g\in\baire\ \exists f\in F:f\not\leq^* g},\\
\dfrak &:=\min\set{|D|}{D\subseteq\baire\setand\forall g\in\baire\ \exists f\in D:g\leq^* f}
\end{align*}
denote the \textit{bounding number} and the \textit{dominating number}, respectively;
and $\cfrak=2^{\aleph_0}$.

Let $\Iwf$ be an ideal of subsets of $X$ such that $\{x\}\in\Iwf$ for all $x\in X$. Throughout this paper, we demand that all ideals satisfy this latter requirement. We recall the following four \emph{cardinal invariants associated with $\Iwf$}:
\begin{align*}
\add(\Iwf)&=\min\largeset{|\Jwf|}{\Jwf\subseteq\Iwf,\ \bigcup\Jwf\notin\Iwf},\\
\cov(\Iwf)&=\min\largeset{|\Jwf|}{\Jwf\subseteq\Iwf,\ \bigcup\Jwf=X},\\
\non(\Iwf)&=\min\set{|A|}{A\subseteq X,\ A\notin\Iwf},\quad\text{and}\\
\cof(\Iwf)&=\min\set{|\Jwf|}{\Jwf\subseteq\Iwf,\ \forall A\in\Iwf\ \exists B\in\Jwf\colon A\subseteq B}.
\end{align*}
These cardinals are referred to as the \emph{additivity, covering, uniformity} and \emph{cofinality of $\Iwf$}, respectively.

Specifically, we are interested in these cardinal invariants for the upcoming ideals:

Let $X$ be a Polish space and $\Bwf(X)$ the $\sigma$-algebra of Borel subsets of $X$.
\begin{enumerate}[label= \rm (\arabic*)]

\item If $X$ is perfect then $\Mwf(X)\eqT \Mwf(\R)$ and $\Cbf_{\Mwf(X)}\eqT \Cbf_{\Mwf(\R)}$, where $\Mwf(X)$ denotes the ideal of meager subsets of $X$ (see~\cite[Ex.~8.32 \&~Thm.~15.10]{Ke2}). Therefore, the cardinal characteristics associated with the meager ideal are independent of the perfect Polish space used to calculate it. When the space is clear from the context, we write $\Mwf$ for the meager ideal.

\item Assume that $\mu\colon\Bwf(X)\to [0,\infty]$ is a $\sigma$-finite measure such that $\mu(X)>0$ and every singleton has measure zero.
Denote by $\Nwf(\mu)$ the ideal generated by the $\mu$-measure zero sets, which is also denoted by $\Nwf(X)$ when the measure on $X$ is clear.
Then $\Nwf(\mu)\eqT \Nwf(\Lb)$ and $\Cbf_{\Nwf(\mu)}\eqT \Cbf_{\Nwf(\Lb)}$ where $\Lb$ denotes the Lebesgue measure on $\R$ (see~\cite[Thm.~17.41]{Ke2}). Therefore, the four cardinal characteristics associated with both measure zero ideals are the same. 

When the measure space is understood, we write $\Nwf$ for the null ideal.

\item For $b\in\baire$ we write $\prod b=\prod_{n<\omega}b(n)$. We denote by $\Ewf(\prod b)$ the ideal generated by the $F_\sigma$ measure zero subsets of $\prod b$. Likewise, define $\Ewf(\R)$ and $\Ewf([0,1])$.
When $\prod b$ is perfect, the map $F_b\colon \prod b\to[0,1]$ defined by
\[F_b(x):=\sum_{n<\omega}\frac{x(n)}{\prod_{i\leq n}b(i)}\]
is a continuous onto function, and it preserves measure. Hence, this map preserves sets between $\Ewf(\prod b)$ and $\Ewf([0,1])$ via images and pre-images. Therefore, $\Ewf(\prod b)\eqT \Ewf([0,1])$ and $\Cbf_{\Ewf(\prod b)}\eqT\Cbf_{\Ewf([0,1])}$. We also have $\Ewf(\R)\eqT \Ewf([0,1])$ and $\Cbf_{\R}\eqT\Cbf_{\Ewf([0,1])}$, as well as $\Ewf(\baire)\eqT \Ewf(2^\omega)$ and $\Cbf_{\Ewf(\baire)}\eqT \Cbf_{\Ewf(2^\omega)}$ (see details in~\cite{GaMe}). It is well-known that $\Ewf\subseteq\Nwf\cap\Mwf$. Even more, it was proved that $\Ewf$ is a proper subideal of $\Nwf\cap\Mwf$ (see~\cite[Lemma 2.6.1]{BJ}). 
\end{enumerate}

When studying cardinal invariants and their order relations, Tukey's connections framework is practical since it simplifies the reasoning and instantly provides results for dual cardinals. For a general overview of this framework, see, e.g.,~\cite{Vojtas,blass}. framework. Here, we mostly follow the notation of the second reference.

We say that $\Rbf=\la X, Y, {\sqsubset}\ra$ is a \textit{relational system} if it consists of two non-empty sets $X$ and $Y$ and a relation $\sqsubset$.
\begin{enumerate}[label=(\arabic*)]
\item A set $F\subseteq X$ is \emph{$\Rbf$-bounded} if $\exists\, y\in Y\ \forall\, x\in F\colon x \sqsubset y$. 
\item A set $E\subseteq Y$ is \emph{$\Rbf$-dominating} if $\forall\, x\in X\ \exists\, y\in E\colon x \sqsubset y$. 
\end{enumerate}
We associate two cardinal invariants with this relational system:
\begin{itemize}
\item[{}] $\bfrak(\Rbf):=\min\set{|F|}{F\subseteq X \text{ is }\Rbf\text{-unbounded}}$ the \emph{bounding number of $\Rbf$}, and

\item[{}] $\dfrak(\Rbf):=\min\set{|D|}{D\subseteq Y \text{ is } \Rbf\text{-dominating}}$ the \emph{dominating number of $\Rbf$}.
\end{itemize}

The \emph{dual of $\Rbf$} is the relational system $\Rbf^\perp:=\la Y,X,{\sqsubset^\perp}\ra$ where $y \sqsubset^\perp x$ means $x \not\sqsubset y$. Note that $\bfrak(\Rbf^\perp)=\dfrak(\Rbf)$ and $\dfrak(\Rbf^\perp)=\bfrak(\Rbf)$.

For relational systems $\Rbf=\la X,Y,{\sqsubset}\ra$ and $\Rbf'=\la X',Y',{\sqsubset'}\ra$, a pair $(\Psi_-,\Psi_+)$ is a \emph{Tukey connection from $\Rbf$ into $\Rbf'$} if 
$\Psi_-\colon X\to X'$ and $\Psi_+\colon Y'\to Y$ are functions satisfying 
\[\forall\, x\in X\ \forall\, y'\in Y'\colon \Psi_-(x) \sqsubset' y' \Rightarrow x \sqsubset \Psi_+(y').\] 
We say that $\Rbf$ is \emph{Tukey below} $\Rbf'$, denoted by $\Rbf\leqT\Rbf'$, if there is a Tukey connection from $\Rbf$ into $\Rbf'$, and we 
say that $\Rbf$ and $\Rbf'$ are \emph{Tukey equivalent}, denoted by $\Rbf\eqT\Rbf'$, if $\Rbf\leqT\Rbf'$ and $\Rbf'\leqT\Rbf$. It is well-known that $\Rbf\leqT\Rbf'$ implies $(\Rbf')^\perp\leqT \Rbf^\perp$, $\bfrak(\Rbf')\leq\bfrak(\Rbf)$ and $\dfrak(\Rbf)\leq\dfrak(\Rbf')$. Hence, $\Rbf\eqT\Rbf'$ implies $\bfrak(\Rbf')=\bfrak(\Rbf)$ and $\dfrak(\Rbf)=\dfrak(\Rbf')$. 

The paper is structured as follows. 

In~\autoref{sec:gen}, we introduce a natural generalization and variants on several cardinal invariants related to $\efrak_n^\const$ and $\vfrak_n^\const$. We also discuss the order relationship among them. 

In~\autoref{sec:idealsconst}, by thinking about the concepts of evading and predicting, we can study their associated $\sigma$-ideals on the reals. 

In~\autoref{sec:furgen}, building upon Kada's~\cite{kadinfgam} research on game theory with predictors, we introduce new invariants that are closely related to $\efrak_n^\const$ and $\vfrak_n^\const$.

In~\autoref{sec:consresults}, we prove several consistency results regarding the previous cardinal invariants. 

 In~\autoref{sec:openQ}, we delineates some open questions.

\section{Towards a general theory of constant evasion and prediction}\label{sec:gen}

This section aims to generalize the notions of constant evasion and constant prediction defined in \autoref{intro} and look at the corresponding cardinals. 

\begin{notation}
For $b\in\baire[(\omega+1)]$ with $b(n)>0$ for every $n\in\omega$. We fix the following terminology
\begin{enumerate}
\item Let 
\begin{align*}
&\xprod b=\set{f\in\baire}{\forall^\infty n\in\omega\ f(n)<b(n)}, \\
\qquad\Seq_n(b)=\prod_{k<n}b(k)&, \qquad
\Seq_{< k}(b)=\bigcup_{n\leq k}\Seq_n(b),
\qquad\Seq(b)=\bigcup_{n<\omega}\Seq_n(b).
\end{align*} 
\item A~function $\sigma\colon\Seq(b)\to\omega$ is called a~\emph{predictor} on~$\prod b$.
\item A~predictor $\sigma\colon\Seq(b)\to\omega$ is a~\emph{limited predictor}
on~$\prod b$, if $\sigma(s)\in b(n)$ for every $n<\omega$ and $s\in\Seq_n(b)$.
Let $\Sigma_b$ denote the set of limited predictors on~$\prod b$.
\end{enumerate} 
In particular, if $b(n)=K$ for $n<\omega$ where $2\leq K\leq\omega$,
then we write $\baire[K]$ and~$\Sigma_K$ instead of $\prod b$ and~$\Sigma_b$,
respectively;
$\Sigma_\omega$~is the set of predictors on~$\baire$.
\end{notation}

One way of generalizing the notions of constant evasion and constant prediction goes as follows.

\begin{definition}\label{def:pred}
Let $b\in\baire[(\omega+1\smallsetminus2)]$ and let $\rel$ be
a sequence ${\rel}=\seq{{\rel_n}}{n\in\omega}$ of
binary relations ${\rel}_n\subseteq b(n)\times b(n)$, $n\in\omega$, such that 
$\emptyset\ne{\rel}_n^{-1}[\{m\}]\ne b(n)$ for every $m\in b(n)$.
\begin{enumerate}
\item For $f\in\prod b$ and $\sigma\in\Sigma_b$, say
$\sigma$~\emph{constantly $\rel$-predicts}~$f$, written as $f\predby^\cpr_{\rel}\sigma$ 
\begin{align*}
f\predby^\cpr_{\rel}\sigma&\text{ iff }
\exists k\in\omega\ \forall i\in\omega\ \exists j\in[i,i+k)\
f(j)\rel_j\sigma(f\frestr j).
\\*
&\text{ iff }
\exists k\in\omega\ \forall^\infty i\in\omega\ \exists j\in[i,i+k)\
f(j)\rel_j\sigma(f\frestr j)
\end{align*}

\item We consider the following cardinal invariants related to the relational systems of
constant $\rel$-prediction 
\begin{align*}
\efrak^\const_{b,\rel}
&:=\min\set{|F|}{F\subseteq\prod b\setand\neg\exists\sigma\in\Sigma_b\ \forall f\in F\colon f\predby^\cpr_{\rel}\sigma},\\
\vfrak^\const_{b,\rel}
&:=\min\set{|S|}{S\subseteq\Sigma_b\setand\forall f\in\prod b\ \exists\sigma\in S\colon f\predby^\cpr_{\rel}\sigma}.
\end{align*}

\item Define the relational system $\Ebf^\cpr_{b,\rel}=\la\prod b,\Sigma_b,{\predby^\cpr_{\rel}}\ra$. Hence, $\efrak^\const_{b,\rel}=\bfrak(\Ebf^\cpr_{b,\rel})$ and $\vfrak^\const_{b,\rel}=\dfrak(\Ebf^\cpr_{b,\rel})$.
\end{enumerate}

In some cases, we simplify the notation
$\efrak^\const_{b,\rel}$ and $\vfrak^\const_{b,\rel}$ as follows:
If $b(n)=K$ for $n<\omega$, then we replace the index~$b$ with~$K$;
we omit the index~$K$ if $K=\omega$;
and we omit the index~$\rel$ if $\rel$ is~$=$. Likewise for $\Ebf^\cpr_{b,\rel}$. 
\end{definition}

In particular, we are interested in these cardinal invariants 
\begin{align*}
 &\efrak^\const_b=\bfrak\left(\Ebf^\cpr_{b}\right),
&&\vfrak^\const_b=\dfrak\left(\Ebf^\cpr_{b}\right),\\
 &\efrak^\const_{*b}=\bfrak\left(\xprod b,\Sigma_\omega,{\predby^\cpr_{=}}\right),
&&\vfrak^\const_{*b}=\dfrak\left(\xprod b,\Sigma_\omega,{\predby^\cpr_{=}}\right),\\
 &\efrak^\const_K=\bfrak(\Ebf^\cpr_{K}),
&&\vfrak^\const_K=\dfrak(\Ebf^\cpr_{K}),\\
 &\efrak^\const=\bfrak(\Ebf^\cpr),
&&\vfrak^\const=\dfrak(\Ebf^\cpr),\\
 &\efrak^\const_2=\bfrak(\Ebf^\cpr_2),
&&\vfrak^\const_2=\dfrak(\Ebf^\cpr_2),\\
 &\efrak^\const_{b,\ne}=\bfrak\left(\Ebf^\cpr_{b,\ne}\right),
&&\vfrak^\const_{b,\ne}=\dfrak\left(\Ebf^\cpr_{b,\ne}\right),\\
 &\efrak^\const_{K,\ne}=\bfrak\left(\Ebf^\cpr_{K,\ne}\right),
&&\vfrak^\const_{K,\ne}=\dfrak\left(\Ebf^\cpr_{K,\ne}\right),\\
 &\efrak^\const_{\ne}=\bfrak\left(\Ebf^\cpr_{\ne}\right),
&&\vfrak^\const_{\ne}=\dfrak\left(\Ebf^\cpr_{\ne}\right),\\*
 &\efrak^\const_\leq=\bfrak(\Ebf^\cpr_{\leq}),
&&\vfrak^\const_\leq=\dfrak(\Ebf^\cpr_{\leq}).
\end{align*}
In~\cite{BreIII}, the relation $x\not\predby^\cpr_{\leq}\sigma$ was studied and was read as ``$x$ strongly evades~$\sigma$'' but not its cardinals were studied. 

According to \autoref{b0}~\ref{b0:2},
$\la\xprod b,\Sigma_\omega,{\predby^\cpr_{=}}\ra\eqT\Ebf_{b}^\cpr$ and
therefore
$\efrak^\const_{*b}=\efrak^\const_b$ and $\vfrak^\const_{*b}=\vfrak^\const_b$.
On the other hand, according to \autoref{b3}, the analogous
cardinal invariants $\efrak^\const_{*b,\ne}$ and $\vfrak^\const_{*b,\ne}$ are
not defined.
It may be reasonable to work only with limited predictors to avoid trivial prediction relations.

\begin{notation}
Denote ${\equiv_b}=\seq{{\equiv_{b(n)}}}{n<\omega}$,
${\not\equiv_b}=\seq{{\not\equiv_{b(n)}}}{n<\omega}$, and
${\leq_b}=\seq{{\leq}\cap{\equiv_{b(n)}}}{n\in\omega}$
where for $0<m\leq\omega$, $i\equiv_m j$ if $i\equiv j\mod m$.
\end{notation}

Note that ${\equiv_\omega}$ and ${\leq_\omega}$ coincide with~$=$;
${\not\equiv_\omega}$ coincides with~$\ne$;
$\leq_1$ coincides with~$\le$.
Also note that the relational systems $\la\prod b,\Sigma_b,{\predby^{A}_{=}}\ra$ and
$\la\prod b,\Sigma_b,{\predby^{A}_{\neq}}\ra$ are instances of $\la\prod b,\Sigma_b,{\predby^{A}_{\equiv_d}}\ra$ and
$\la\prod b,\Sigma_b,{\predby^{A}_{\not\equiv_d}}\ra$
for $d=b$; always we can assume that $d\leq b$.

We will need the projection $\pi_b$ of functions, the extension $e_b$ and the restriction $r_b$ of predictors $\pi_b\colon{}^{<\omega}\omega\cup\baire\to\Seq(b)\cup\prod b$, $e_b\colon\Sigma_b\to\Sigma_\omega$, $r_b\colon\Sigma_\omega\to\Sigma_b$, defined by
\begin{align*}
&\pi_b(f)(j)=(f(j)\bmod b(j)),
&&\text{$f\in{}^{<\omega}\omega\cup\baire$ and $j\in\dom(f)$,}\\
&r_b(\tau)(t)=(\tau(t)\bmod b(|t|)),
&&\text{$\tau\in\Sigma_\omega$ and $t\in\Seq(b)$},\\
&e_b(\sigma)(s)=\sigma(\pi_b(s)),
&&\text{$\sigma\in\Sigma_b$ and $s\in{}^{<\omega}\omega$}
\end{align*}
(here, $(m\bmod\omega)=m$ and $(m\bmod1)=0$ for $m\in\omega$).

The study of relational
systems for the relation~$\predby^\cpr_{=}$ can be reduced to systems with
codomain~$\Sigma_\omega$ thanks to the following lemma.

\begin{lemma}\label{b0}
Let $b',b\in\baire[(\omega+1\smallsetminus2)]$, $b'\leq^*b$, and\/ 
$\prod b'\subseteq F\subseteq\xprod b'$
\begin{enumerate}[label=\rm(\arabic*)]
\item\label{b0:0} $\Ebf^\cpr_{b}\eqT\Ebf^\cpr$ and $\Ebf^\cpr_{b,\ne}\eqT\Ebf^\cpr_{\ne}$
\item\label{b0:1}
$\Ebf^\cpr_{b}\eqT
\la\prod b,\Sigma_\omega,{\predby^{\cpr}_{=}}\ra$ and\/ $\la\prod b,\Sigma_\omega,{\predby^{\cpr}_{\rel}}\ra\leqT
\Ebf^\cpr_{b,\rel}$.
\item\label{b0:2}
$\Ebf^\cpr_{b'}\eqT
\la F,\Sigma_\omega,{\predby^{\cpr}_{=}}\ra\leqT
\Ebf^\cpr_{b}$.
\item\label{b0:3}
$\la\prod b',\Sigma_\omega,{\predby^{\cpr}_{\neq}}\ra\leqT
\la F,\Sigma_\omega,{\predby^{\cpr}_{\neq}}\ra\leqT
\Ebf^\cpr_{b,\ne}\leqT
\Ebf^\cpr_{b',\ne}$.
\end{enumerate}
\end{lemma}

\begin{proof}
\ref{b0:0}--\ref{b0:1}:
The pairs of functions
\begin{align*}
&(\id,r_b)\colon
\Ebf^\cpr_{b}\to\la\baire,\Sigma_\omega,{\predby^{\cpr}_{\equiv_b}}\ra,
&&(\pi_b\frestr\baire,e_b)\colon
\la\baire,\Sigma_\omega,{\predby^{\cpr}_{\equiv_b}}\ra\to\Ebf^\cpr_{b},
\\
&(\id,r_b)\colon
\Ebf^\cpr_{b,\ne}\to\la\baire,\Sigma_\omega,{\predby^{\cpr}_{\not\equiv_b}}\ra,
&&(\pi_b\frestr\baire,e_b)\colon
\la\baire,\Sigma_\omega,{\predby^{\cpr}_{\not\equiv_b}}\ra\to\Ebf^\cpr_{b,\ne},\\
&(\id,r_b)\colon
\Ebf^\cpr_{b}\to\la\prod b,\Sigma_\omega,{\predby^{\cpr}_{=}}\ra,
&&(\id,e_b)\colon
\la\prod b,\Sigma_\omega,{\predby^{\cpr}_{\rel}}\ra\to\Ebf^\cpr_{b,\rel},
\end{align*}
where $\id$ denotes the identity on~$\prod b$,
are Tukey connections
because for every $f\in\prod b$, $g\in\baire$, $\sigma\in\Sigma_b$, $\tau\in\Sigma_\omega$,
\begin{align*}
&f\predby^{\cpr}_{\equiv_b}\tau\Rightarrow f\predby^{\cpr}_{=}r_b(\tau),
&&\pi_b(g)\predby^{\cpr}_{=}\sigma\Rightarrow g\predby^{\cpr}_{\equiv_b}e_b(\sigma),\\
&f\predby^{\cpr}_{\not\equiv_b}\tau\Rightarrow f\predby^{\cpr}_{\neq}r_b(\tau),
&&\pi_b(g)\predby^{\cpr}_{\neq}\sigma\Rightarrow g\predby^{\cpr}_{\not\equiv_b}e_b(\sigma),\\
&f\predby^{\cpr}_{=}\tau\Rightarrow f\predby^{\cpr}_{=}r_b(\tau),
&&f\predby^{\cpr}_{\rel}\sigma\Rightarrow f\predby^{\cpr}_{\rel}e_b(\sigma).
\end{align*}

\ref{b0:2}--\ref{b0:3}:
Thanks to \ref{b0:1} we have
$\la\prod b',\Sigma_{b'},{\predby^{\cpr}_{=}}\ra\eqT\la\prod b',\Sigma_\omega,{\predby^{\cpr}_{=}}\ra$.
Therefore, it is enough to show (for any~$\rel$)
\begin{enumerate}[label=\rm$\bullet_\arabic*$]
\item\label{tk:1} $\la\prod b',\Sigma_\omega,{\predby^{\cpr}_{\rel}}\ra\leqT
\la F,\Sigma_\omega,{\predby^{\cpr}_{\rel}}\ra$,
\item\label{tk:2} $\la F,\Sigma_\omega,{\predby^{\cpr}_{=}}\ra\leqT\Ebf^\cpr_{b'}$,
\item\label{tk:3} $\la F,\Sigma_\omega,{\predby^{\cpr}_{\rel}}\ra\leqT\Ebf^\cpr_{b,\rel}$, and 
\item\label{tk:4} $\Ebf^\cpr_{b,\rel}\leqT
\Ebf^\cpr_{b',\rel}$.
\end{enumerate}
\eqref{tk:1} is witnessed by inclusion maps and \eqref{tk:2}~is an instance of~\eqref{tk:3} for $b=b'$.
We prove~\eqref{tk:3} and~\eqref{tk:4} by showing that the pairs of functions
\begin{enumerate}[label=\rm($\boxdot_\arabic*$)]
\item\label{ptk:1} $(\pi_b\frestr F,e_b)\colon\la F,\Sigma_\omega,{\predby^{\cpr}_{\rel}}\ra\to
\Ebf^\cpr_{b,\rel}$, 
\item\label{ptk:2} $(\pi_{b'}\frestr\prod b,e_{b',b})\colon
\Ebf^\cpr_{b,\ne}\to
\Ebf^\cpr_{b',\ne}$.
\end{enumerate}

where $e_{b',b}(\tau)=r_b(e_{b'}(\tau))$ for $\tau\in\Sigma_{b'}$, are Tukey connections. 

To see~\ref{ptk:1}, let $f\in F$ and $\sigma\in\Sigma_b$ be such that
$\pi_b(f)\predby^{\cpr}_{\rel}\sigma$. Then for some $k\in\omega$,
\begin{align*}
\forall i\in\omega\bsp\exists j\in[i,i+k)\
\pi_b(f)(j)&\rel_j\sigma(\pi_b(f)\frestr j)=
\sigma(\pi_b(f\frestr j))=e_b(\sigma)(f\frestr j).
\end{align*}
Since $\forall^\infty j\in\omega\bsp f(j)<b'(j)\leq b(j)$, then
$\forall^\infty j\in\omega\bsp\pi_b(f)(j)=f(j)\in b(j)$ and thus
\[
\forall^\infty i\in\omega\bsp\exists j\in[i,i+k)\
f(j)\rel_j e_b(\sigma)(f\frestr j),
\]
and hence, $f\predby^{\cpr}_{\rel}e_b(\sigma)$. Lastly, to show~\ref{ptk:2}, let $g\in\prod b$ and $\tau\in\Sigma_{b'}$ be such that
$\pi_{b'}(g)\predby^{A}_{\neq}\tau$.
Then for some $k\in\omega$,
\begin{align*}
\forall i\in\omega\bsp\exists j\in[i,i+k)\
\pi_{b'}(g)(j)&\ne\tau(\pi_{b'}(g)\frestr j)=
\tau(\pi_{b'}(g\frestr j)).
\end{align*}
Since $b'\leq^*b$, $\forall^\infty j\in\omega\bsp\tau(\pi_{b'}(g\frestr j))=e_{b'}(\tau)(g\frestr j)=e_{b',b}(\tau)(g\frestr j)$, and hence,
\[
\forall^\infty i\in\omega\bsp\exists j\in[i,i+k)\
(g(j)\bmod b'(j))\ne e_{b',b}(\tau)(g\frestr j).
\]
Since the values $e_{b',b}(\tau)(g\frestr j)$ are smaller than~$b'(j)$, we get
\[
\forall^\infty i\in\omega\bsp\exists j\in[i,i+k)\
g(j)\ne e_{b',b}(\tau)(g\frestr j)
\]
and hence, $g\predby^{\cpr}_{\neq}e_{b',b}(\tau)$.
\end{proof}

\begin{remark}\label{b3}
\startlist
\begin{enumerate}[label=\rm(\arabic*)]
\item If $b'\in\baire[(\omega\smallsetminus2)]$ and $F\subseteq\xprod b'$, then in \autoref{b0}~\ref{b0:3}
moreover, trivially, $\la\prod b',\Sigma_\omega,{\predby^\cpr_{\neq}}\ra\eqT
\la F,\Sigma_\omega,{\predby^\cpr_{\neq}}\ra$.
This holds because there is $\sigma\in\Sigma_\omega$ such that
$f\predby^\cpr_{\neq}\sigma$ for every $f\leq^*b'$:
Consider $\sigma\in\Sigma_\omega$ defined by
$\sigma(s)=b(|s|)$ for $s\in{}^{<\omega}\omega$.
Then $\sigma$~is a~${\predby^\cpr_{\neq}}$-dominating predictor, i.e., for every $f\in\xprod b'$,
$\forall^\infty j\in\omega\bsp f(j)\ne\sigma(f\frestr j)$.
\item If $b\in\baire[(\omega\smallsetminus2)]$, then there is a~${\predby^\cpr_{\leq}}$~dominating predictor. Therefore, the relational system $\la\prod b,\Sigma_b,{\predby^\cpr_{\leq}}\ra$ is not interesting.
\end{enumerate}
\end{remark}

\begin{lemma}\label{b0x}
Let $b\in(\omega+1\smallsetminus2)^\omega$ and\/
$\rel$ satisfy\/ \ref{*0} and\/~\ref{*1} where 
\begin{enumerate}[label=\rm($*_{\arabic*}$)]\setcounter{enumii}{-1}
\item\label{*0}
$\forall n\in\omega\bsp\forall m\in b(n)\bsp{\rel}_n^{-1}[\{m\}]\ne b(n)$
(the dominating number of~$\rel_n$ is~$>1$).
\item\label{*1}
$\forall n\in\omega\bsp\dom({\rel_n})=b(n)$
(the bounding number of~$\rel_n$ is~$>1$).
\end{enumerate}
Then
\begin{enumerate}[label=\rm(\arabic*)]
\item\label{neq-r-eq}
$\Ebf^\cpr_{b,\ne}\leqT\Ebf^\cpr_{b,\rel}\leqT
\Ebf^\cpr_{b}$.
\item\label{neq=eq}
$\Ebf^\cpr_{2,\ne}\eqT
\Ebf^\cpr_{2,\rel}\eqT
\Ebf^\cpr_{2,=}$ provided that $b(n)=2$ for all $n\in\omega$.
\end{enumerate}
\end{lemma}
\begin{proof}
\ref{neq-r-eq}:
By \ref{*0} and \ref{*1}, respectively, there are functions
$\varphi^0_n,\varphi^1_n:b(n)\to b(n)$ such that
$\neg(\varphi^{0}_n(m)\rel_nm)$ and $m\rel_n\varphi^1_n(m)$ whenever $m\in b(n)$.
Define $\Psi^0_+,\Psi^1_+:\Sigma_b\to\Sigma_b$ by
$\Psi^i_+(\sigma)(s)=\varphi^i_n(\sigma(s))$ whenever $s\in\Seq_n(b)$, $i=0$,~$1$.
Then
\[(\id_{\prod b},\Psi^0_+)\colon
\Ebf^\cpr_{b,\ne}\to
\Ebf^\cpr_{b,\rel}
\setand
(\id_{\prod b},\Psi^1_+)\colon
\Ebf^\cpr_{b,\rel}\to
\Ebf^\cpr_{b}\]
are Tukey connections because for every $f\in\prod b$ and $\sigma\in\Sigma_b$,
$f\predby^{\cpr}_{\rel}\sigma\Rightarrow f\predby^{\cpr}_{\neq}\Psi^0_+(\sigma)$
and
$f\predby^{\cpr}_{=}\sigma\Rightarrow f\predby^{\cpr}_{\rel}\Psi^1_+(\sigma)$
(because $k\rel_nm\Rightarrow k\neq\varphi^0_n(m)$ and $k=n\Rightarrow k\rel_n\varphi^1_n(m)$).

\ref{neq=eq}:
By \ref{neq-r-eq} the Tukey connection
$(\id_{2^\omega},\Psi_+)\colon\Ebf^\cpr_{2}\to
\Ebf^\cpr_{2,\ne}$ suffices where
$\Psi_+(\sigma)(s)=1-\sigma(s)$ for $\sigma\in\Sigma_2$ and $s\in2^{<\omega}$.
\end{proof}

\begin{definition}
Let $\efrak^\const_{\leq_b}=\bfrak(\la\baire,\Sigma_\omega,{\predby^\cpr_{\leq_b}}\ra)$ and
$\vfrak^\const_{\leq_b}=\dfrak(\la\baire,\Sigma_\omega,{\predby^\cpr_{\leq_b}}\ra)$ for $b\in\baire[(\omega+1\smallsetminus1)]$.
\end{definition}

By virtue of~\autoref{b0}~\ref{b0:2}--\ref{b0:3}, for every $b,b'\in\baire[(\omega\smallsetminus2)]$, $b'\leq^*b$
implies
$\efrak^\const_{b}\leq\efrak^\const_{b'}$,
$\vfrak^\const_{b'}\leq\vfrak^\const_{b'}$,
$\efrak^\const_{b',\ne}\leq\efrak^\const_{b',\ne}$, and
$\vfrak^\const_{b,\ne}\leq\efrak^\const_{b',\ne}$.
Therefore we define:

\begin{definition}
\begin{align*}
& &\efrak^\const_\ubd&=\min\set{\efrak^\const_b}{2\leq b\in\baire}, \quad
&& &\vfrak^\const_\ubd&=\sup\set{\vfrak^\const_b}{2\leq b\in\baire},\\
& &\efrak^\const_{\ubd,\ne}&=\sup\set{\efrak^\const_{b,\ne}}{2\leq b\in\baire}, \quad
&& &\vfrak^\const_{\ubd,\ne}&=\min\set{\vfrak^\const_{b,\ne}}{2\leq b\in\baire},\\
& &\efrak^\const_\hyp&=\min\set{\efrak^\const_K}{2\leq K<\omega}, \quad
&&&\vfrak^\const_\hyp&=\sup\set{\vfrak^\const_K}{2\leq K<\omega},\\
& &\efrak^\const_{\hyp,\ne}&=\sup\set{\efrak^\const_{K,\ne}}{2\leq K<\omega}, \quad
&& &\vfrak^\const_{\hyp,\ne}&=\min\set{\vfrak^\const_{K,\ne}}{2\leq K<\omega}.
\end{align*}
\end{definition}

\section{Constant evasion ideals and duality}\label{sec:idealsconst}

In this section, we will associate ideals on the reals regarding the concepts of evading and predicting as follows.

Let $b$ and ${\rel}$ be as in~\autoref{def:pred}.
Denote
\begin{align*}
A_{b,\rel}^{\sigma,k}&=\set{x\in\prod b}{\forall i\in\omega\ \exists j\in[i,i+k)\
x(j)\rel_j\sigma(x\frestr j)},
\quad\sigma\in\Sigma_b,\ k\in\omega,\\
\Iwf^{\const,0}_{b,\rel}&=\set{X\subseteq\prod b}{
\exists\sigma\in\Sigma_b\ \exists k\in\omega\
X\subseteq A_{b,\rel}^{\sigma,k}},\\
\Iwf^\const_{b,\rel}&=\set{X\subseteq\prod b}
{\exists\sigma\in\Sigma_b\
X\subseteq\bigcup_{k\in\omega}A_{b,\rel}^{\sigma,k}}.
\end{align*}
Like in~\autoref{def:pred}, if $b(n)=K$ for $n<\omega$, then we replace the index~$b$
with~$K$; we omit the index~$K$, if $K=\omega$; and we omit the index~$\rel$,
if ${\rel}$ is~$=$.

In \cite{kamoeva}, the $\sigma$-ideal generated by $\Iwf^\const_{K}$ is studied. 

\begin{lemma}\label{b1}
Let $b\in\baire[(\omega+1\smallsetminus2)]$ and\/ $2\leq K<\omega$.
\begin{enumerate}[label=\rm(\alph*)]
\item\label{b1:a}
$\Iwf^{\const,0}_{b,\rel}\subseteq\Iwf^\const_{b,\rel}$
are ideals on\/~$\prod b$.
\item\label{b1:b}
$\Iwf^{\const,0}_b=\set{X\cap\prod b}{X\in\Iwf^{\const,0}}$ and
$\Iwf^\const_b=\set{X\cap\prod b}{X\in\Iwf^\const}$.
\item\label{b1:c}
$\Iwf^\const_{b,\rel}\subseteq\Mwf(\prod b)$ and
$\Iwf^\const_{K,\rel}\subseteq\Ewf(\baire[K])$.
\item\label{b1:d}
$\Iwf^\const_b\subseteq\Iwf^\const_{b,\ne}$ and
$\Iwf^\const\subseteq\Iwf^\const_\leq\subseteq\Iwf^\const_{\ne}$.
\item\label{b1:e}
$\non(\Iwf^\const_{b,\rel})=\efrak^\const_{b,\rel}$ and\/
$\cov(\Iwf^\const_{b,\rel})=\vfrak^\const_{b,\rel}$.
\end{enumerate}
\end{lemma}

\begin{proof}
\ref{b1:a}:
$\Iwf^{\const,0}_{b,\rel}$, $\Iwf^\const_{b,\rel}$ are ideals because
$A_{b,\rel}^{\sigma_0,k_0}\cup A_{b,\rel}^{\sigma_1,k_1}\subseteq
A_{b,\rel}^{\sigma,3k}$ where
$k=\max\{k_0,k_1\}$ and
$\sigma(s)=\sigma_0(s)$, if $|s|\in\bigcup_{n\in\omega}[2nk,2nk+k)$, and
$\sigma(s)=\sigma_1(s)$, otherwise.

\ref{b1:b}:
Similarly like in \autoref{b0}~\ref{b0:1},
for $\sigma\in\Sigma_b$ and $\tau\in\Sigma_\omega$,
$A_{b,=}^{\sigma,k}\subseteq A_{\omega,=}^{e_b(\sigma),k}\cap\prod b$ and
$A_{\omega,=}^{\tau,k}\cap\prod b\subseteq A_{b,=}^{r_b(\tau),k}$.

\ref{b1:c}:
The sets $A_{b,\rel}^{\sigma,k}$ are closed nowhere dense sets in~$\prod b$
and the sets $A_{K,\rel}^{\sigma,k}$ are closed null subsets of~$\baire[K]$
for $2\leq K<\omega$ because ${\rel}_n^{-1}[\{m\}]\ne b(n)$
(see~\cite{kamoeva} in case of $\Iwf^\const_K$).
Therefore $\Iwf^\const_{b,\rel}\subseteq\Mwf(\prod b)$ and
$\Iwf^\const_{K,\rel}\subseteq\Ewf(\baire[K])$.

\ref{b1:d}:
For $\sigma\in\Sigma_b$ (possibly $b=\omega$) let $\sigma\oplus 1\in\Sigma_b$
be defined by $(\sigma\oplus 1)(s)=(\sigma(s)+\nobreak1 \bmod b(|s|))$.
The inclusions
$\Iwf^\const_b\subseteq\Iwf^\const_{b,\ne}$ and
$\Iwf^\const\subseteq\Iwf^\const_\leq\subseteq\Iwf^\const_{\ne}$ hold
because
$A_{b,=}^{\sigma,k}\subseteq A_{b,\ne}^{\sigma\oplus1,k}$ and
$A_{\omega,=}^{\sigma,k}\subseteq A_{\omega,\leq}^{\sigma,k}\subseteq
A_{\omega,\ne}^{\sigma\oplus1,k}$.

\ref{b1:e}:
By definitions.
\end{proof}

\begin{lemma}\label{b2}
Let $b\in\baire[(\omega\smallsetminus2)]$ be such that $b'\leq^*b$ and
let\/ $2\leq K<\omega$.
\begin{enumerate}[label=\rm(\arabic*)]
\item\label{b2:a}
\begin{enumerate}[label=\rm(\alph*)]
\item\label{b2:a:1}
$\min\{\efrak^\const_\ubd,\bfrak\}\leq\efrak^\const\leq\efrak^\const_\ubd\leq
\efrak^\const_b\leq\efrak^\const_{b'}\leq\efrak^\const_2$.
\item\label{b2:a:2}
$\efrak^\const_2=\efrak^\const_{2,\ne}\leq\efrak^\const_{K,\ne}\leq\non(\Ewf)$.
\item\label{b2:a:3}
$\efrak^\const_{b',\ne}\leq\efrak^\const_{b,\ne}\leq\efrak^\const_{\ubd,\ne}\leq\efrak^\const_{\ne}\leq\non(\Mwf)$.
\item\label{b2:a:4}
$\efrak^\const\leq\efrak^\const_\leq\leq\efrak^\const_{\ne}$.
\item\label{b2:a:5}
$\max\{\efrak^\const,\bfrak\}\leq\efrak^\const_\leq\leq\min\{\dfrak,\non(\Mwf)\}$.
\end{enumerate}
\item\label{b2:b}
\begin{enumerate}[label=\rm(\alph*)]
\item\label{b2:b:1}
$\vfrak^\const_2\leq\vfrak^\const_{b'}\leq\vfrak^\const_b\leq
\vfrak^\const_\ubd\leq\vfrak^\const\leq\max\{\vfrak^\const_\ubd,\dfrak\}$.
\item\label{b2:b:2}
$\cov(\Ewf)\leq\vfrak^\const_{K,\ne}\leq\vfrak^\const_{2,\ne}=\vfrak^\const_2$.
\item\label{b2:b:3}
$\cov(\Mwf)\leq\vfrak^\const_{\ne}\leq\vfrak^\const_{\ubd,\ne}\leq\vfrak^\const_{b,\ne}\leq\vfrak^\const_{b',\ne}$.
\item\label{b2:b:4}
$\vfrak^\const_{\ne}\leq\vfrak^\const_\leq\leq\vfrak^\const$.
\item\label{b2:b:5}
$\max\{\cov(\Mwf),\bfrak\}\leq\vfrak^\const_\leq\leq\min\{\vfrak^\const,\dfrak\}$.
\end{enumerate}
\end{enumerate}
\end{lemma}

\begin{proof}
All inequalities except the following are obvious consequences of
\autoref{b0} and \autoref{b1} (or easy to see).

$\min\{\efrak^\const_\ubd,\bfrak\}\leq\efrak^\const$:
Let $\kappa<\min\{\efrak^\const_\ubd,\bfrak\}$.
Since $\kappa<\bfrak$ for every set $F\in[\baire]^\kappa$ there is
$b\in\baire$ such that $F\subseteq\xprod b$ and since
$\kappa<\efrak^\const_\ubd\leq\efrak^\const_b=\efrak^\const_{*b}$ there is
$\sigma\in\Sigma_\omega$ such that $f\predby^\cpr_{=}\sigma$ for all $f\in F$.
This proves that $\kappa<\efrak^\const$.

$\vfrak^\const\leq\max\{\vfrak^\const_\ubd,\dfrak\}$:
Let $D\subseteq\baire$ be a~dominating family of functions of
cardinality~$\dfrak$ and for every $b\in D$ let $V_b\subseteq\Sigma_\omega$
be a~$\predby^\cpr_{=}$-dominating set for $\prod b$ of cardinality $\vfrak^\const_b$.
Then $\bigcup_{b\in D}V_b$ is a~$\predby^\cpr_{=}$-dominating set for~$\baire$
of cardinality~$\leq\max\{\vfrak^\const_\ubd,\dfrak\}$.

$\efrak^\const_\leq\leq\dfrak$:
For $\sigma\in\Sigma_\omega$ define
$f_\sigma(j)=\sigma(f_\sigma\frestr j)+1$.
Let $F\subseteq\Sigma_\omega$ and $|F|<\efrak^\const_\leq$.
Then there is $\tau\in\Sigma_\omega$ such that for every $\sigma\in F$,
$f_\sigma(j)\leq\tau(f_\sigma\frestr j)$ for infinitely
many~$j$ and hence
$\sigma(s)<\tau(s)$ for infinitely many $s\in{}^{<\omega}\omega$.
Therefore $F$~is not a~dominating family.

$\bfrak\leq\vfrak^\const_\leq$:
Let $F\subseteq\Sigma_\omega$ and $|F|<\bfrak$.
There is $\tau\in\Sigma_\omega$ such that $\forall\sigma\in F$
$\forall^\infty s\bsp\sigma(s)<\tau(s)$.
Define $f_\tau(j)=\sigma(f_\tau\frestr j)$.
For every $\sigma\in F\bsp\forall^\infty j\in\omega$
$\sigma(f_\tau\frestr j)<\tau(f_\tau\frestr j)=f_\tau(j)$, i.e.,
$f_\tau$ is not constantly $\leq$-predicted by any $\sigma\in F$.
\end{proof}

In addition to the above lemma, the following results are known. 

\begin{lemma}
\startlist
\begin{enumerate}[label=\rm(\arabic*)]
\item \emph{\cite{kamoeva}} $\efrak^\const\leq\cov(\Mwf)$ and\/ $\non(\Mwf)\leq\vfrak^\const$.

\item \emph{\cite{BreGarV}} $\vfrak_2^\const\geq\non(\Mwf)$. In particular, $\vfrak_2^\const\geq\efrak_2^\const$.

\end{enumerate}
\end{lemma}

\section{Further generalizations: \texorpdfstring{$\efrak_{b,\rel}^\forall$}{} and \texorpdfstring{$\vfrak_{b,\rel}^\forall$}{} }\label{sec:furgen}

In this last section, inspired by Kada's~\cite{kadinfgam} work on the characterizations of some cardinals of Cichoń's diagram in terms of infinite games via predictors, we present the following generalization:

\begin{definition}
Let $b$ and ${\rel}$ be as in~\autoref{def:pred}.
\begin{enumerate}
\item For $f\in\prod b$ and $\sigma\in\Sigma_b$ say	that
$\sigma$~\emph{$\rel$-predicts}~$f$, written as $f\predby^\spr_{\rel}\sigma$ 
\[
f\predby^\spr_{\rel}\sigma\text{ iff }
\forall^\infty j\ f(j)\rel_j\sigma(f\frestr j).
\]
\item We consider two cardinal invariants related to the relational systems of
$\rel$-prediction 
\begin{align*}
\efrak_{b,\rel}^\forall
&:=\min\set{|F|}{F\subseteq\prod b\setand\neg\exists \sigma\in\Sigma_b\ \forall f\in F\colon f\predby^\spr_{\rel}\sigma},\\
\vfrak_{b,\rel}^\forall
&:=\min\set{|S|}{S\subseteq\Sigma_b\setand\forall f\in\prod b\ \exists \sigma\in S\colon f\predby^\spr_{\rel}\sigma}.
\end{align*}
\item Define the relational system $\Ebf^\spr_{b,\rel}=\la\prod b,\Sigma_b,{\predby^\spr_{\rel}}\ra$. Hence, $\efrak_{b,\rel}^\forall=\bfrak(\Ebf^\spr_{b,\rel})$ and $\vfrak_{b,\rel}^\forall=\dfrak(\Ebf^\spr_{b,\rel})$.
\end{enumerate}
Like in~\autoref{def:pred}, if $b(n)=K$ for $n<\omega$, then we replace the index~$b$
with~$K$; we omit the index~$K$, if $K=\omega$; and we omit the index~$\rel$,
if ${\rel}$ is~$=$.
\end{definition}

\begin{lemma}\label{cp-pr}
$\Ebf^\cpr_{b,\rel}\leqT
\Ebf^\spr_{b,\rel}$.
As a consequence, $\efrak_{b,\rel}^\forall\le\efrak^\const_{b,\rel}$ and\/
$\vfrak^\const_{b,\rel}\le\vfrak_{b,\rel}^\forall$.
\end{lemma}

\begin{proof}
Because $f\predby^\spr_{\rel}\sigma$ implies $f\predby^\cpr_{\rel}\sigma$.
\end{proof}

\begin{lemma}\label{leq_b}
\startlist
\begin{enumerate}[label=\rm(\arabic*)]
\item
$\efrak^\const_{\leq_b}\leq\min\{\efrak^\const_\leq,\efrak^\const_b\}$
and\/
$\max\{\vfrak^\const_\leq,\vfrak^\const_b\}\leq\vfrak^\const_{\leq_b}$. 
\item\label{leq_b.2}
$\efrak_{\leq_b}^\forall=\min\{\efrak_\leq^\forall,\efrak_b^\forall\}$ and\/
$\max\{\vfrak_\leq^\forall,\vfrak_b^\forall\}=\vfrak_{\leq_b}^\forall$.
\end{enumerate}
\end{lemma}

\begin{proof}
All inequalities $\leq$ hold because $({\leq_b})_n\subseteq{\leq}\cap{\equiv_{b(n)}}$
and because by \autoref{b0},
$\la\prod b,\Sigma_b,{\predby^\cpr_{=}}\ra\eqT
\la\baire,\Sigma_\omega,{\predby^\cpr_{\equiv_b}}\ra$
and
$\la\prod b,\Sigma_b,{\predby^\spr_{=}}\ra\eqT
\la\baire,\Sigma_\omega,{\predby^\spr_{\equiv_b}}\ra$. The proof of the latter Tukey identity is the same as the proof of the former one.
The inequalities~$\geq$ in~\ref{leq_b.2} hold because
${\predby^\spr_{\leq_b}}\supseteq{\predby^\spr_{\leq}}\cap{\predby^\spr_{\equiv_b}}$.
\end{proof}

We first introduce more terminology.

\begin{definition}\label{def:Ed}
Let $b:=\seq{b(n)}{n<\omega}$ be a sequence of non-empty sets. Define the relational system $\Ed_b:=\la\prod b,\prod b,{\neq^\infty}\ra$ where $x=^\infty y$ means $x(n)=y(n)$ for infinitely many $n$. The relation $x\neq^\infty y$ expresses that \emph{$x$ and $y$ are eventually different}. We just write $\Ed:=\Ed_\omega$ (when $b$ is the constant function $\omega$). Denote $\balc_{b,1}:=\bfrak(\Ed_b)$ and $\dalc_{b,1}:=\dfrak(\Ed_b)$.
\end{definition}

\begin{lemma}\label{lem:Ed}
$\Ed_b\leqT\Ebf^\spr_{b,\ne}$. In particular, $\efrak_{b,\ne}^\forall\leq\balc_{b,1}$ and\/ $\dalc_{b,1}\leq\vfrak_{b,\ne}^\forall$.
\end{lemma}
\begin{proof}
Define $\Psi_-\colon\prod b\to\prod b$, defined as the identity map and $\Psi_+\colon\Sigma_b\to\prod b$ is defined as follows: For each $\sigma\in\Sigma_b$ define $\Psi_+(\sigma)=f_\sigma\in\prod b$ recursively by letting $f_\sigma(n)=\sigma(f_\sigma\frestr n)$ for all $n$. Note that it is clear that
$(\Psi_-,\Psi_+)$ is the required Tukey connection.
\end{proof}

Based on Kada's~\cite{kadinfgam} work on game-theoretic characterizations of the uniformity and
covering of the meager ideal and bounding and dominating number in terms of predictors, we can get combinatorial characterizations of some cardinal invariants in Cicho\'n's diagram:

\begin{lemma}\label{thm:kada}
\startlist
\begin{enumerate}[label=\rm(\arabic*)]
\item\label{thm:kada:a} 
$\efrak_{\neq}^\forall=\non(\Mwf)$ and\/ 
$\vfrak_{\neq}^\forall=\cov(\Mwf)$.
\item\label{thm:kada:b} 
$\efrak_{\leq}^\forall=\bfrak$ and\/ 
$\vfrak_{\leq}^\forall=\dfrak$.
\end{enumerate}
\end{lemma}

A new combinatorial characterization of the uniformity and coverage of $\Mwf$ is obtained below. 

\begin{lemma}\label{newchM}
$\efrak^\const_{\neq}=\non(\Mwf)$
and\/
$\vfrak^\const_{\neq}=\cov(\Mwf)$.
\end{lemma}
\begin{proof}
We know $\non(\Mwf)\le\efrak^\const_{b,\neq}$ and\/
$\vfrak^\const_{b,\neq}\le\cov(\Mwf)$ by~\autoref{cp-pr} and \autoref{thm:kada} \ref{thm:kada:a}; and since $\efrak^\const_{\neq}\leq\non(\Mwf)$
and\/
$\cov(\Mwf)\leq\vfrak^\const_{\neq}$, we then conclude $\efrak^\const_{\neq}=\non(\Mwf)$
and\/
$\vfrak^\const_{\neq}=\cov(\Mwf)$.
\end{proof}

We could ask if $\efrak_{\leq}^\const=\bfrak$ and $\vfrak_{\leq}^\const=\dfrak$? By ~\autoref{cp-pr} and \autoref{thm:kada} \ref{thm:kada:b}, we know $\bfrak\leq\efrak_{\leq}^\const$ and $\vfrak_{\leq}^\const\leq\dfrak$, but the converse does hold because it is known the consistency of $\efrak^\const>\bfrak$ and $\vfrak^\const>\dfrak$ (see~\cite[Lem.~3.6]{BreIII}); and since $\efrak_\leq^\const\geq\efrak^\const$ and $\vfrak^\const\geq\vfrak_\leq^\const$ by \autoref{b2} \ref{b2:a}, we derive the consistency of $\efrak_\leq^\const>\bfrak$ and $\vfrak^\const_\leq>\dfrak$.

\autoref{thm:kada} \ref{thm:kada:b} and \autoref{newchM} yield:

\begin{corollary}
$\add(\Mwf)=\min\{\efrak_\le^\forall,\vfrak_{\ne}^\const\}$ and\/ $\cof(\Mwf)=\min\{\vfrak_\le^\forall,\efrak_{\ne}^\const\}$.
\end{corollary}

\begin{corollary}
$\efrak_{\leq_b}=\min\{\bfrak,\efrak_b^\forall\}$ and\/
$\vfrak_{\leq_b}=\max\{\dfrak,\vfrak_b^\forall\}$.
\end{corollary}

\begin{proof}
By~\autoref{thm:kada}, \autoref{b2}, \autoref{cp-pr} for $b(n)=\omega$, and
\autoref{leq_b}~\ref{leq_b.2}.
\end{proof}

\begin{definition}
\begin{align*}
& &\efrak_\ubd^\forall&= \min\set{\efrak_b^\forall}{b\in\baire}, \quad
&& &\vfrak_\ubd^\forall&= \sup\set{\vfrak_b^\forall}{b\in\baire},\\
& &\efrak_{\ubd,\ne}^\forall&= \sup\set{\efrak_{b,\ne}^\forall}{b\in\baire}, \quad
&& &\vfrak_{\ubd,\ne}^\forall&= \min\set{\vfrak_{b,\ne}^\forall}{b\in\baire},\\
& &\efrak_\hyp^\forall&= \min\set{\efrak_K^\forall}{2\leq K<\omega}, \quad
&&&\vfrak_\hyp^\forall &= \sup\set{\vfrak_K^\forall}{2\leq K<\omega},\\
& &\efrak_{\hyp,\ne}^\forall&= \sup\set{\efrak_{K,\ne}^\forall}{2\leq K<\omega}, \quad
&& &\vfrak_{\hyp,\ne}^\forall&= \min\set{\vfrak_{K,\ne}^\forall}{2\leq K<\omega}.
\end{align*}
\end{definition}

Similarly to~\autoref{b2}, we derive:

\begin{lemma}
\startlist
\begin{enumerate}[label=\rm(\arabic*)]
\item
\begin{enumerate}[label=\rm(\alph*)]
\item 
$\efrak_{2}^\forall=\efrak_{2,\neq}^\forall\leq\efrak_{b,\neq}^\forall\leq\efrak_{\ubd,\neq}^\forall\leq\efrak_{K,\neq}^\forall\leq\efrak_{\hyp,\neq}^\forall\leq\efrak_{\neq}^\forall$. 

\item
$\vfrak_{\neq}^\forall\leq\vfrak_{\hyp,\neq}^\forall\leq\vfrak_{K,\neq}^\forall\leq\vfrak_{\ubd,\neq}^\forall\leq\vfrak_{b,\neq}^\forall\leq\vfrak_{2,\neq}^\forall=\vfrak_{2}^\forall$. 
\end{enumerate}
\item 
\begin{enumerate}[label=\rm(\alph*)]
\item
$\efrak_{\hyp,\neq}^\forall\leq \efrak_{\hyp,\neq}^\const\leq\non(\Ewf)$, $\efrak_{b,\neq}^\forall\leq\efrak_{b,\neq}^\const$, and\/ $\efrak_{\ubd,\neq}\leq\efrak_{\ubd,\neq}^\const$.
\item $\cov(\Ewf)\leq\vfrak_{\hyp,\neq}^\const\leq \vfrak_{\hyp,\neq}^\forall$, $\vfrak_{b,\neq}^\const\leq\vfrak_{b,\neq}^\forall$, and\/ $\vfrak_{\ubd,\neq}^\const\leq\vfrak_{\ubd,\neq}^\forall$.
\end{enumerate}
\end{enumerate}
\end{lemma}

\autoref{cichonext} illustrates all the inequalities shown as yet.

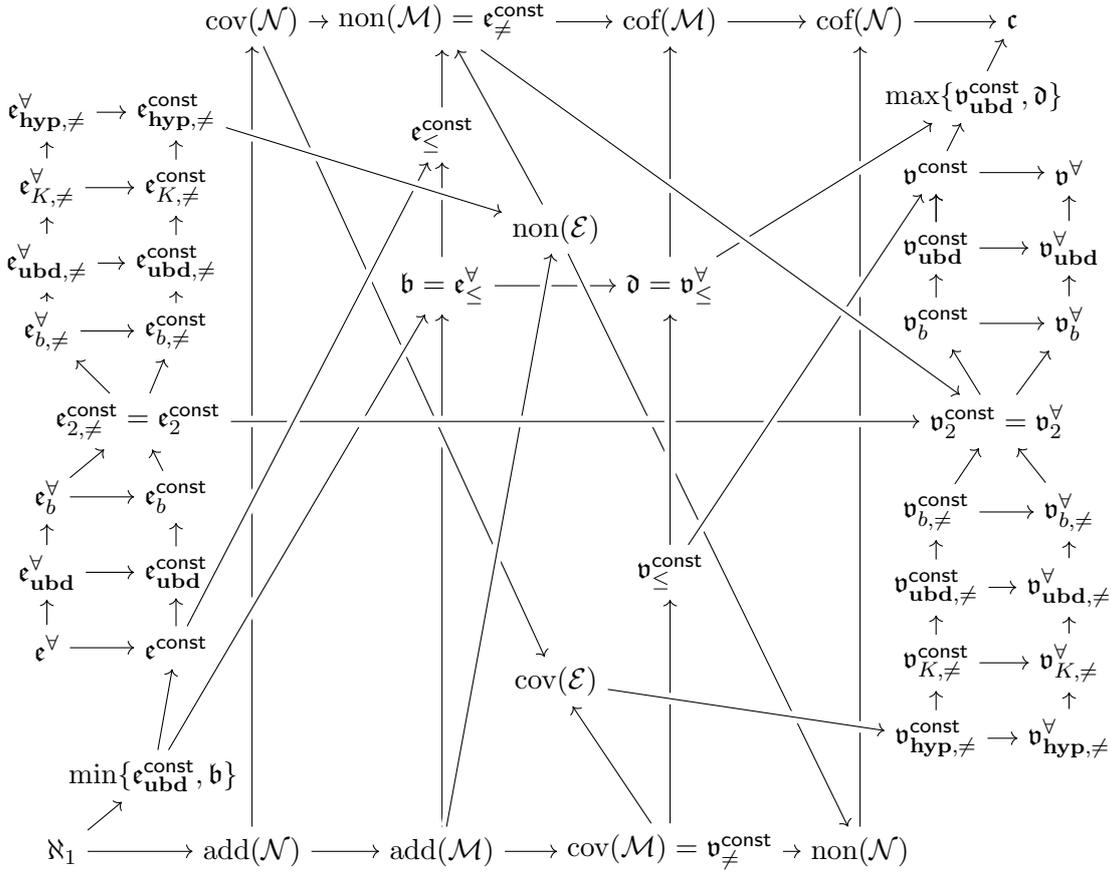
\begin{figure}[!ht]
\centering
\begin{tikzpicture}[scale=1]
\small{
\node (aleph1) at (-1,2) {$\aleph_1$};
\node (addn) at (1.5,2){$\add(\Nwf)$};
\node (covn) at (1.5,13){$\cov(\Nwf)$};
\node (nonn) at (9.5,2) {$\non(\Nwf)$};
\node (cfn) at (9.5,13) {$\cof(\Nwf)$};
\node (addm) at (4,2) {$\add(\Mwf)$};
\node (covm) at (7,2) {$\cov(\Mwf)=\vfrak^\const_{\neq}$};
\node (nonm) at (4,13) {$\non(\Mwf)=\efrak^\const_{\neq}$};
\node (cfm) at (7,13) {$\cof(\Mwf)$};
\node (b) at (4,9.5) {$\bfrak=\efrak_{\leq}^\forall$};
\node (econstlq) at (4,11.5) {$\efrak^\const_\leq$};
\node (d) at (7,9.5) {$\dfrak=\vfrak_{\leq}^\forall$};
\node (vconstleq) at (7,5.7) {$\vfrak^\const_{\leq}$};
\node (cove) at (5.5,4.25) {$\cov(\Ewf)$};
\node (none) at (5.5,10.25) {$\non(\Ewf)$};
\node (c) at (11.5,13) {$\cfrak$};

\node (mineubdb) at (.2, 3){$\min\{\efrak^\const_\ubd,\bfrak\}$};
\node (econst) at (.5, 4.7) {$\efrak^\const$};
\node (e) at (-1.2, 4.7) {$\efrak^\forall$};

\node (econstubd) at (.5, 5.7) {$\efrak^\const_\ubd$};
\node (econstb) at (.5, 6.7) {$\efrak^\const_b$};
\node (econst2) at (0, 7.7) {$\efrak^\const_{2,\ne}=\efrak^\const_2$};
\node (econstKneq) at (.5, 10.8) {$\efrak^\const_{K,\neq}$};
\node (econsthypneq) at (.5, 11.8) {$\efrak^\const_{\hyp,\ne}$};
\node (ebconstneq) at (.5, 8.9) {$\efrak^\const_{b,\ne}$};
\node (econstubdneq) at (0.5,9.8) {$\efrak^\const_{\ubd,\neq}$};


\node (eubd) at (-1.2, 5.7) {$\efrak_\ubd^\forall$};
\node (eb) at (-1.2, 6.7) {$\efrak_b^\forall$};
\node (ebneq) at (-1.2, 8.9) {$\efrak_{b,\ne}^\forall$};
\node (eKneq) at (-1.2, 10.8) {$\efrak_{K,\neq}^\forall$};
\node (eubdneq) at (-1.2,9.8) {$\efrak_{\ubd,\neq}^\forall$};
\node (ehypneq) at (-1.2, 11.8) {$\efrak_{\hyp,\ne}^\forall$};

\node (vhypneq) at (12.25, 3.5) {$\vfrak_{\hyp,\ne}^\forall$};
\node (vKneq) at (12.25, 4.5) {$\vfrak_{K,\ne}^\forall$};
\node (vubdneq) at (12.25, 5.5) {$\vfrak_{\ubd,\ne}^\forall$};
\node (vbneq) at (12.25, 6.5) {$\vfrak_{b,\ne}^\forall$};
\node (vb) at (12.25, 9) {$\vfrak_{b}^\forall$};
\node (vubd) at (12.25,10) {$\vfrak_{\ubd}^\forall$};
\node (v) at (12.25, 11) {$\vfrak^\forall$};

\node (maxeubdb) at (11,12){$\max\{\vfrak^\const_\ubd,\dfrak\}$};
\node (vconst) at (10.5,11) {$\vfrak^\const$};
\node (vconstubd) at (10.5,10) {$\vfrak^\const_\ubd$};
\node (vconstb) at (10.5,9) {$\vfrak^\const_b$};
\node (vconst2) at (11.3,7.7) {$\vfrak^\const_2=\vfrak_{2}^\forall$};
\node (vconstKneq) at (10.5, 4.5) {$\vfrak^\const_{K,\neq}$};
\node (vconstubdneq) at (10.5,5.5) {$\vfrak^\const_{\ubd,\neq}$};
\node (vconsthypneq) at (10.5, 3.5) {$\vfrak^\const_{\hyp,\ne}$};
\node (vconstbneq) at (10.5, 6.5) {$\vfrak^\const_{b,\ne}$};

\draw (vubdneq) edge[->] (vbneq);
\draw (vKneq) edge[->] (vubdneq);
\draw (vhypneq) edge[->] (vKneq);
\draw (vbneq) edge[->] (vconst2);
\draw (vconst2) edge[->] (vb);
\draw (vb) edge[->] (vubd);
\draw (vubd) edge[->] (v);

\draw (vconstubdneq) edge[->] (vconstbneq);
\draw (vconstKneq) edge[->] (vconstubdneq);
\draw (vconsthypneq) edge[->] (vconstKneq);
\draw (vconstbneq) edge[->] (vconst2);

\draw (vconsthypneq) edge[->] (vhypneq);
\draw (vconstKneq) edge[->] (vKneq);
\draw (vconstubdneq) edge[->] (vubdneq);
\draw (vconstbneq) edge[->] (vbneq);
\draw (vconstb) edge[->] (vb);
\draw (vconstubd) edge[->] (vubd);
\draw (vconst) edge[->] (v);

\draw (eubdneq) edge[->] (econstubdneq);
\draw (eKneq) edge[->] (econstKneq);
\draw (eb) edge[->] (econstb);
\draw (eubd) edge[->] (econstubd);
\draw (e) edge[->] (econst);
\draw (ehypneq) edge[->] (econsthypneq);
\draw (ebneq) edge[->] (ebconstneq);

\draw (aleph1) edge[->] (addn);
\draw(addn) edge[->] (covn)
      (covn) edge [->] (nonm)
      (covm) edge [->] (nonn)
      (nonm)edge [->] (cfm)
      (cfm)edge [->] (cfn)
      (cfn) edge[->] (c);

\draw
   (addn) edge [->]  (addm)
   (addm) edge [->]  (covm)
   (nonn) edge [->]  (cfn);
\draw (addm) edge [->] (b)
      (b)  edge [->] (econstlq);

\draw     (d)  edge[->] (cfm);
\draw (b) edge [->] (d);

\draw (none) edge[line width=.15cm,white,-]  (nonn);
\draw (none) edge [->]  (nonn);
\draw (none) edge [->]  (nonm);
\draw (cove) edge [<-]  (covm);

\draw  (cove) edge[line width=.15cm,white,-] (covn);
\draw  (cove) edge [<-]  (covn);


\draw  (addm) edge[line width=.15cm,white,-] (none);
\draw  (addm) edge [->]  (none);

\draw  (mineubdb) edge[line width=.15cm,white,-] (b);
\draw  (mineubdb) edge [->]  (b);

\draw  (mineubdb) edge[line width=.15cm,white,-](econst);
\draw  (mineubdb) edge [->]  (econst);

\draw  (econst2) edge[line width=.15cm,white,-](vconst2);
\draw  (econst2) edge [->] (vconst2);

\draw  (econst) edge[line width=.15cm,white,-](econstlq);
\draw  (econst) edge [->] (econstlq);

\draw  (vconst) edge[line width=.15cm,white,-](vconstleq);
\draw  (vconst) edge [<-] (vconstleq);

\draw  (d) edge[line width=.15cm,white,-](maxeubdb);
\draw  (d) edge [->] (maxeubdb);

\draw(econst) edge[->] (econstubd);
\draw(econstubd) edge[->] (econstb);
\draw(econstb) edge[->] (econst2);
\draw(econst2) edge[->] (ebconstneq);
\draw(econstKneq) edge[<-] (econstubdneq);
\draw(ebconstneq) edge[->] (econstubdneq);
\draw(econstKneq) edge[->] (econsthypneq);

\draw(e) edge[->] (eubd);
\draw(eubd) edge[->] (eb);
\draw(eb) edge[->] (econst2);
\draw(econst2) edge[->] (ebneq);
\draw(eKneq) edge[<-] (eubdneq);
\draw(ebneq) edge[->] (eubdneq);
\draw(eKneq) edge[->] (ehypneq);


\draw(vconst) edge[->] (maxeubdb);
\draw(vconst) edge[<-] (vconstubd);
\draw(vconst) edge[<-] (vconstubd);
\draw(vconstubd) edge[<-] (vconstb);
\draw(vconstb) edge[<-] (vconst2);


\draw  (vconstleq) edge[line width=.15cm,white,-](d);
\draw  (vconstleq) edge [->] (d);

\draw  (covm) edge[line width=.15cm,white,-](vconstleq);
\draw  (covm) edge [->] (vconstleq);

\draw  (econsthypneq) edge[line width=.15cm,white,-](none);
\draw  (econsthypneq) edge [->] (none);

\draw  (vconsthypneq) edge[line width=.15cm,white,-](cove);
\draw  (vconsthypneq) edge [<-] (cove);

\draw  (vconst2) edge[line width=.15cm,white,-](nonm);
\draw  (vconst2) edge [<-] (nonm);

\draw  (econstlq) edge[line width=.15cm,white,-](nonm);
\draw  (econstlq) edge [->] (nonm);

\draw (aleph1) edge[->] (mineubdb);

\draw (maxeubdb) edge[->] (c);

}
\end{tikzpicture}
\caption{Cicho\'n's diagram along with the cardinal invariants introduced so far.}
\label{cichonext}
\end{figure}

\section{Consistency results}\label{sec:consresults}

In the following, let us present some known consistent results regarding~\autoref {cichonext}. 
\begin{enumerate}[label=\rm(C\arabic*)]
\item\label{c:mds} Assume $\nu$ and $\lambda$ cardinals such that $\aleph_1\leq\nu\leq\lambda=\lambda^{<\nu}$. Brendle~\cite{BreIII} constructed a FS iteration to force $\bfrak=\nu<\efrak^\const=\lambda$.

\item Kamo~\cite{kamopred} built a CS iteration to force $\vfrak^\const=\aleph_1<\dfrak=\aleph_2$. 

\item With $\nu$ and $\lambda$ as in~\ref{c:mds}, Brendle performed a FS iteration to force $\efrak^\const_\leq=\nu<\efrak^\const_2=\lambda$.

\item With $\nu$ and $\lambda$ as in~\ref{c:mds}, Brendle and Shelah~\cite{BreShevaIV} built a FS iteration of amoeba forcing of $\nu\lambda$ to force $\efrak_2^\const=\nu<\add(\Nwf)=\lambda$.

\item Kamo~\cite{kamoeva} built a CS iteration to force $\cof(\Nwf)=\aleph_1<\vfrak^\const_2=\aleph_2$.

\item\label{c1:mds} Assume $\nu$ and $\lambda$ cardinals such that $\aleph_1\leq\nu\leq\lambda=\lambda^{\aleph_0}$. The FS iteration of length $\lambda\nu$ of random forcing
forces $\non(\Ewf)=\bfrak=\aleph_1\leq\cov(\Nwf)=\non(\Nwf)=\nu\leq\cov(\Ewf)=\cfrak=\lambda$ (see e.g~\cite[Thm.~5.4]{Car4E}). Then by~\autoref{b2}~\ref{b2:a}, we get $\efrak^\const_{\hyp,\neq}=\aleph_1$ and $\vfrak^\const_{\hyp,\neq}=\lambda$.

\item With $\nu$ and $\lambda$ as in~\ref{c1:mds}. The FS iteration of length $\lambda\nu$ of Hechler forcing forces $\cov(\Nwf)=\aleph_1\leq\bfrak=\nu\leq\cfrak=\lambda$. Thanks to~\autoref{c11}, we get $\efrak_2^\const=\aleph_1$. 
\end{enumerate}

While we have exhibited some consistent results related to our newly defined cardinal invariants, which are, in fact, new; we will follow this trend and provide some new consistent results about them.

We start by recalling the preservation properties introduced by Judah and Shelah~\cite{JS} and Brendle~\cite{Br} for FS iterations of ccc posets, which were extended in~\cite[Sect.~4]{CM} by the first author and Mej\'ia, which we will use it.

\begin{definition}\label{b8}
Let $\Rbf=\la X,Y,{\sqsubset}\ra$ be a relational system, let $\theta$ be a cardinal and let $M$ be a~set

\begin{enumerate}[label=\rm(\alph*)]
\item An object $y\in Y$ is \textit{$\Rbf$-dominating over $M$} if $x\sqsubset y$ for all $x\in X\cap M$.

\item An object $x\in X$ is \textit{$\Rbf$-unbounded over $M$} if it $\Rbf^\perp$-dominating over $M$, that is, $x\not\sqsubset y$ for all $y\in Y\cap M$.
\end{enumerate}

\end{definition}

By establishing a Tukey connection between their relational systems and certain simple relational systems, such as $\Cv_{[\lambda]^{<\theta}}$ and $[\lambda]^{<\theta}$ for cardinals $\theta\leq\lambda$ with uncountable regular~$\theta$, we demonstrate in our forcing constructions that specific cardinal invariants assume particular values within a generic extension

For instance, if $\Rbf$ is a relational system and we force $\Rbf\eqT\Cv_{[\lambda]^{<\theta}}$, then we obtain $\bfrak(\Rbf)=\non([\lambda]^{<\theta})=\theta$ and $\dfrak(\Rbf)=\cov([\lambda]^{<\theta})=\lambda$, the latter when either $\theta$ is regular or $\lambda>\theta$.
This discussion motivates the following characterizations of the Tukey order between $\Cv_{[X]^{<\theta}}$ and other relational systems.

\begin{lemma}[{\cite[Lem.~1.16]{CM22}}]\label{b12}
Let\/ $\Rbf=\la X,Y,{\sqsubset}\ra$ be a relational system, $\theta$ be an infinite cardinal, and $I$ be a set of size ${\geq}\theta$. Then $\bfrak(\Rbf)\geq\theta$ iff\/ $\Rbf\leqT\Cv_{[X]^{<\theta}}$.
\end{lemma}

We examine the types of relational systems that are well-defined.

\begin{definition}\label{b11}
Say that $\Rbf=\la X,Y,{\sqsubset}\ra$ is a \textit{Polish relational system (Prs)} if
\begin{enumerate}[label=\rm(\arabic*)]
\item\label{b11:1}
$X$ is a Perfect Polish space,
\item\label{b11:2}
$Y$ is a non-empty analytic subspace of some Polish $Z$, and
\item\label{b11:3}
${\sqsubset}=\bigcup_{n<\omega}{\sqsubset_{n}}$ where $\seq{{\sqsubset_{n}}}{n\in\omega}$ is some increasing sequence of closed subsets of $X\times Z$ such that, for any $n<\omega$ and for any $y\in Y$,
$({\sqsubset_{n}})^{y}=\set{x\in X}{x\sqsubset_{n}y}$ is closed nowhere dense.
\end{enumerate}
\end{definition}

\begin{remark}\label{b4}
By~\autoref{b11}~\ref{b11:3}, $\la X,\Mwf(X),{\in}\ra$ is Tukey below $\Rbf$ where $\Mwf(X)$ denotes the $\sigma$-ideal of meager subsets of $X$. Therefore, $\bfrak(\Rbf)\leq \non(\Mwf)$ and $\cov(\Mwf)\leq\dfrak(\Rbf)$.
\end{remark}

For the rest of this section, fix a Prs $\Rbf=\la X,Y,{\sqsubset}\ra$ and an infinite cardinal $\theta$.

\begin{definition}[Judah and Shelah {\cite{JS}}, Brendle~{\cite{Br}}]\label{c12}
A poset $\Por$ is \textit{$\theta$-$\Rbf$-good} if, for any $\Por$-name $\dot{h}$ for a member of $Y$, there is a nonempty set $H\subseteq Y$ (in the ground model) of size ${<}\theta$ such that, for any $x\in X$, if $x$ is $\Rbf$-unbounded over $H$ then $\Vdash x\not\sqsubset \dot{h}$.

We say that $\Por$ is \textit{$\Rbf$-good} if it is $\aleph_1$-$\Rbf$-good.
\end{definition}

The previous is a standard property associated with preserving $\bfrak(\Rbf)$ small and $\dfrak(\Rbf)$ large after forcing extensions.

\begin{remark}
Notice that $\theta<\theta_0$
implies that any $\theta$-$\Rbf$-good poset is $\theta_0$-$\Rbf$-good. Also, if $\Por \lessdot\Qor$ and $\Qor$ is $\theta$-$\Rbf$-good, then $\Por$ is $\theta$-$\Rbf$-good.
\end{remark}

\begin{lemma}[{\cite[Lemma~2.7]{CM}}]\label{b6}
Assume that $\theta$ is a regular cardinal. Then any poset of size ${<}\theta$
is $\theta$-$\Rbf$-good. In particular, Cohen forcing\/ $\Cor$ is\/ $\Rbf$-good.
\end{lemma}

FS iterations add Cohen reals, sometimes considered a limitation of the method. 

Let $\Mbf$ denote the Polish relational system from \cite[Example~3.8(6)]{CMR2}
such that $\Mbf\eqT\Cbf_\Mwf$.

\begin{lemma}[{\cite[Cor.~2.7]{CM22}}]\label{Cohenlimit}
Let $\pi$ be a limit ordinal of uncountable cofinality and let\/
$\Por_\pi=\seq{\Por_\alpha,\Qnm_\alpha}{\alpha<\pi}$ be a FS iteration of non-trivial posets.
If\/ $\Por_\pi$ is\/ $\cf(\pi)$-cc then it forces $\pi\leqT\Mbf$.
In particular, $\Por_\pi$ forces $\non(\Mwf)\leq\cf(\pi)\leq\cov(\Mwf)$. 
\end{lemma}

The impact of adding cofinally many $\Rbf$-dominating reals along an FS iteration is illustrated by the next result

\begin{lemma}[{\cite[Lemma~2.9]{CM}}]\label{thm:fullgen}
Let\/ $\Rbf$ be a definable relational system of the reals, and let $\nu$ be a limit ordinal of uncountable cofinality.
If\/ $\Por_\nu=\seq{\Por_\xi,\Qnm_\xi}{\xi<\nu}$ is a FS iteration of $\cf(\nu)$-cc posets that adds\/ $\Rbf$-dominating reals cofinally often, then\/ $\Por_\nu$ forces\/ $\Rbf\leqT\nu$. 

In addition, if\/ $\Rbf$ is a Prs and all iterands are non-trivial, then\/ $\Por_\nu$ forces\/ $\Rbf\eqT\Mbf\eqT\nu$.
In particular, $\Por_\nu$ forces\/ $\bfrak(\Rbf)=\dfrak(\Rbf)=\non(\Mwf)=\cov(\Mwf) =\cf(\nu)$.
\end{lemma}

As FS iterations along, good posets are preserved in the following way.

\begin{theorem}[{\cite[Thm.~4.11]{BCM2}}]\label{Comgood}
Let $\theta\geq\aleph_1$ be a regular cardinal and let\/ $\seq{\Por_\xi,\Qnm_\xi}{\xi<\pi}$ be a finite support iteration such that\/  $\Por_\xi$ forces that\/ $\Qnm_\xi$ is a non-trivial $\theta$-cc $\theta$-$\Rbf$-good poset for $\xi<\pi$.  Then\/ $\Por_\pi$ is $\theta$-$\Rbf$-good. Moreover, 
if $\pi\geq\theta$ then\/ $\Por_\pi$ forces\/ $\Cv_{[\pi]^{<\theta}}\leqT \Rbf$, which implies\/ $\bfrak(\Rbf)\leq\theta$ and\/ $|\pi|\leq\dfrak(\Rbf)$. 
\end{theorem}

\subsection*{Forcing \texorpdfstring{$\efrak_\leq^\const$}{} small}\

Drawing inspiration from the proof of the~\cite[Crucial Lemma 4.1]{BreIII}, we present the following property:

\begin{definition}\label{c0} Let $\Por$ be a forcing notion. 
\begin{enumerate}[label=\rm(\arabic*)]
\item Let $A,B\subseteq\Por$.
We say that $B$~refines~$A$, if $\forall p\in B\bsp\exists q\in A\bsp p\leq q$.
Denote
$A^{\not\perp}=\set{p\in\Por}{\exists q\in A\bsp p\parallel q}$.

\item\label{c0:2} A~set $Q\subseteq\Por$ is said to be $\compact$-linked, if there is a~set
$Q'\subseteq\Por$
(we may require $Q'\subseteq Q^{\not\perp}$)
with the following two properties:
\begin{enumerate}
\item
$\exists A\in[Q']^{<\omega}\bsp Q\subseteq A^{\not\perp}$.
\item
If $A\in[Q']^{<\omega}$ and $Q\subseteq A^{\not\perp}$, then
for every open dense set $D\subseteq\Por$ there is $B\in[D\cap Q']^{<\omega}$
such that $B$~refines~$A$ and $Q\subseteq B^{\not\perp}$.
\end{enumerate}

\item For an infinite cardinal $\theta$, $\Por$ is \emph{$\theta$-$\compact$-linked} if $\Por=\bigcup_{\alpha<\theta}Q_\alpha$ for some $\compact$-linked $Q_\alpha$ ($\alpha<\theta$). When $\theta=\aleph_{0}$, we write $\sigma$-$\compact$-linked.
\end{enumerate}
\end{definition}

Below,  it is the first instance of a $\compact$-linked:

\begin{lemma}\label{c1}
Every finite set is\/ $\compact$-linked. In particular, 
Cohen forcing is $\sigma$-$\compact$-linked.
\end{lemma}

\begin{proof}
If $Q\subseteq\Por$ is finite, consider
$Q'=\set{p\in\Por}{\exists q\in Q\bsp p\leq q}$.
\end{proof}

Our $\compact$-linkednes is connected to $\Fr$-linked property due to Mej\'ia~\cite{mejvert}. To illustrate this, we begin with the following notation: 

\begin{itemize}
\item Denote by $\Fr:=\set{\omega\menos a}{a\in[\omega]^{<\aleph_0}}$ the \emph{Fr\'echet filter}.

\item A filter $F$ on $\omega$ is \emph{free} if $\Fr\subseteq F$. A set $x\subseteq\omega$ is \emph{$F$-positive} if it intersects every member of $F$. Denote by $F^+$ the family of $F$-positive sets. Note that $x\in\Fr^+$ iff $x$ is an infinite subset of $\omega$.
\end{itemize}

\begin{definition}[{\cite{mejvert, BCM}}]\label{Def:Fr}
Let $\Por$ be a poset and $F$ a filter on $\omega$. A set $Q\subseteq \Por$ is \emph{$F$-linked} if, for any $\bar p=\seq{p_n}{n<\omega}\in Q^\omega$, there is some $q\in\Por$ such that \[q\Vdash\set{n\in\omega}{p_n\in\dot G}\in F^+.\]
Observe that, in the case $F=\Fr$, the above equation is “$\set{n\in\omega}{p_n\in\dot G}$ is infinite”. 

We say that $Q$ is \emph{uf-linked (ultrafilter-linked)} if it is $F$-linked for any filter $F$ on $\omega$ containing the \emph{Fr\'echet filter} $\Fr$.

For an infinite cardinal $\theta$, $\Por$ is \emph{$\mu$-$F$-linked} if $\Por=\bigcup_{\alpha<\theta}Q_\alpha$ for some $F$-linked $Q_\alpha$ ($\alpha<\theta$). When these $Q_\alpha$ are uf-linked, we say that $\Por$ is \emph{$\theta$-uf-linked}.
\end{definition}

It is clear that any uf-linked set $Q\subseteq \Por$ is F-linked, and consequently also $\Fr$-linked. But for ccc poset we have:

\begin{lemma}[{\cite[Lem~5.5]{mejvert}}]
If\/ $\Por$ is ccc then any subset of\/ $\Por$ is uf-linked iff it is Fr-linked. 
\end{lemma}

\begin{lemma}\label{c2}
Every\/ $\compact$-linked set is\/ $\Fr$-linked.
\end{lemma}

\begin{proof}
Let $\Vdash_\Por\dot n\in\omega$ and let $Q\subseteq\Por$ be $\compact$-linked.
By condition~\ref{c0:2} of~\autoref{c0} there is a~finite set $B\subseteq\Por$ of
conditions deciding~$\dot n$ such that $Q\subseteq B^{\not\perp}$.
Let $m=\max\set{n\in\omega}{\exists p\in B\bsp p\Vdash\dot n=n}$.
Then $\forall q\in Q\bsp q\not\Vdash m<\dot n$.
\end{proof}

Observe that the converse of the above result does not hold since, a poset $\Por$ can be $\Fr$-linked (\cite[Lem.~3.2]{BreIII}, see also~\cite[Lem.~3.11]{CRS}) but not $\compact$-linked (see below~\autoref{c3}).

\begin{lemma}\label{c3}
If\/ $\Por$ is $\sigma$-$\compact$-linked, then\/ $\Por$ is\/ $\Ebf^\cpr_{\leq}$-good. 
\end{lemma}

\begin{proof}
Let $\Por=\bigcup_kQ_k$ with $Q_k$ is $\compact$-linked.
Fix a~one-to-one enumeration $\set{s_n}{n\in\omega}$ of ${}^{<\omega}\omega$
such that $|s_n|\leq n$ and $s_n\subseteq s_m$ implies $n\leq m$.
Let $D_n=\set{p\in\Por}{\text{$p$ decides $\dot\pi(s_n)$}}$.
By definition of $\compact$-linked sets for every $k\in\omega$ we can find
a~set $Q_k'\subseteq\Por$ and a~sequence of finite sets
$B_{k,n}\subseteq D_n\cap Q_k'$, $n\in\omega$, such that
$Q_k\subseteq B_{k,n}^{\not\perp}$  and $B_{k,n+1}$ refines~$B_{k,n}$.
Define 
\begin{align*}
\psi_k(s_n)&=\max\set{m}{\exists q\in B_{k,n}\bsp
q\Vdash\dot\pi(s_n)=m},\\*
\psi(s_n)&=\max\set{\psi_k(s_n)}{k\leq n}.
\end{align*}

Assume $x\not\predby^\cpr_{\leq}\psi$.
Fix $p\in\Por$ and $k,n\in\omega$.
Then $p\in Q_{k_0}$ for some $k_0\in\omega$.
Let $i\ge\max\{k_0,n\}$ be such that
$\forall j\in[i,i+k)\bsp\psi(x\frestr j)<x(j)$ and let
$x\frestr j=s_{n_j}$.
Then $k_0\leq i\leq n_i\leq\dots<n_{i+k-1}$ and there is $q\in B_{k_0,n_{i+k-1}}$
which is compatible with~$p$.
Let $r$ be a~common extension of $p$ and~$q$.
Then for every $j\in[i,i+k)$, $r$~extends an element of~$B_{k_0,n_j}$ and hence,
$r\Vdash\dot\pi(x\frestr j)\leq
\psi_{k_0}(x\frestr j)\leq\psi(x\frestr j)<x(j)$.
It follows that
$\Vdash``\forall k,n\in\omega\bsp\exists i\ge n\bsp\forall j\in[i,i+k)$
$\dot\pi(x\frestr j)<x(j)$'',
i.e., $\Vdash x\not\predby^\cpr_{\leq}\dot\pi$.
\end{proof}

Now, we introduce a forcing notion that increases $\efrak^\const_{b,\rel}$.

\begin{definition}
Let $b\in\baire$ and ${\rel}\in\{{=},{\leq},{\neq}\}$.
\begin{enumerate}
\item Conditions in the forcing $\Por^b_{\rel}$ are triples $(k,\sigma,F)$ such that
\begin{itemize}
\item
$\set{f\frestr n}{f\in F\setand n\leq k}\subseteq
\dom(\sigma)\subseteq\Seq_{\leq k}(b)$,
we require the equality $\dom(\sigma)=\Seq_{\leq k}(b)$, if $b\in\baire$,
\item
$f\frestr k\ne g\frestr k$ for all $f\ne g$ belonging to~$F$,
and
\item
$f(k)\rel\sigma(f\frestr k)$ for all $f\in F$.
\end{itemize}
\item The partial order is defined by
\begin{align*}
(k',\sigma',F')\leq(k,\sigma,F)\Leftrightarrow{}&
k'\ge k,\ \sigma'\supseteq\sigma,\ F'\supseteq F,\text{ and}\\
&\forall f\in F\ \forall i\in[k,k')\ \exists j\in\{n,n+1\}\
f(j)\rel\sigma'(f\frestr j).
\end{align*}  
\item Let $G\subseteq\Por^b_{\rel}$ be a $\Por^b_{\rel}$-generic over $V$. Then in $V[G]$, the generic predictor is given by
\[\sigma_{\gen,\rel}^b=\bigcup\set{\sigma}{\exists k, F\colon(k,\sigma, F)\in G}.\]
\end{enumerate}
When $\rel$ is $=$, we omit the index in $\Por_{\rel}^b$. 
\end{definition}

Using a standard density argument, it can be proved the following:

\begin{fact}
$\sigma_{\gen,\rel}^b$ is a function from ${}^{<\omega}\omega$ to $\omega$, and and for all $f\in\baire\cap V$, $f\predby^\cpr_{\rel}\sigma_{\gen,\rel}^b$. 
\end{fact}

Just like in~\cite[Lem.~3.2]{BreIII}, we have 

\begin{lemma}\label{lem:Pblink}
$\Por^\const_{b,\rel}$ is $\sigma$-linked for every $b\in\baire[(\omega+1\smallsetminus2)]$.
\end{lemma}

The following shows that $\Por^b$ does not increase $\efrak^\const_\leq$. 

\begin{lemma}\label{c6}
If $b\in\baire$, then\/ $\Por^b$ is $\sigma$-$\compact$-linked.
\end{lemma}

\begin{proof}
Denote
\begin{align*}
&E=\set{(k,\sigma,T)}{\text{$k\in\omega$, $\sigma:\Seq_{\leq k}(b)\to\omega$, and
$T\subseteq\Seq_k(b)$}},\\
&Q_{k,\sigma,T}=\set{(k,\sigma,F)\in\Por^b}{F\frestr k=T\setand|F|=|T|},\quad(k,\sigma,T)\in E.
\end{align*}
The set~$E$ is countable and clearly,
$\Por^b=\bigcup_{(k,\sigma,T)\in E}Q_{k,\sigma,T}$.

We say that $(k',\sigma',F')$ is an almost extension of
$(k,\sigma,F)$ in~$\Por^b$ and we write 
$(k',\sigma',F')\leq_a(k,\sigma,F)$, if
\begin{enumerate}
\item[(i)]
$\exists G\subseteq F'$
$G\frestr k'=F\frestr k'$ and $(k',\sigma',F')\leq(k,\sigma,G)$
and
\item[(ii)]
$\forall f\in F\bsp\exists j\in\{k'-1,k'\}\bsp\sigma'(f\frestr j)=f(j)$;
this is unspoken in~\cite{BreIII}.
\end{enumerate}
It is easy to see that
$(k',\sigma',F')\leq(k,\sigma,F)$ implies $(k',\sigma',F')\leq_a(k,\sigma,F)$.

\begin{clm}\label{c7}
If\/ $(k',\sigma',F')\leq_a(k,\sigma,F)$, then $(k',\sigma',F')$ and
$(k,\sigma,F)$ are compatible.
\end{clm}

\begin{proof}
Clearly, $k'\ge k$ and $f\frestr k'\ne g\frestr k'$ for
all distinct $f,g\in F'$ and for all distinct $f,g\in F$.
If $k'=k$, then $\sigma'=\sigma$ and the set
\[\set{p\in\Por^b}{k_p=k\setand\sigma_p=\sigma}\] is $2$-linked (this can proved exactly like in~\autoref{lem:Pblink}).
If $k'>k$ the proof is similar:
Find $k''\ge k'$ such that $f\frestr k''\ne g\frestr k''$ for
all distinct $f,g\in F'\cup F$.
We extend~$\sigma'$ to $\sigma'':\Seq_{\leq k''}(b)\to\omega$ so that
$(k'',\sigma'',F'\cup F)\in\Por^b$ and
$(k'',\sigma'',F'\cup F)$ extends $(k,\sigma,F)$ and $(k',\sigma',F')$.
Note that for every $f\in F$, $\forall i\in[k,k')$
$\exists j\in\{i,i+1\}\bsp\sigma'(f\frestr j)=f(j)$;
this holds by~(i) for every $i\in[k,k'-2]$, and by~(ii) for $i=k'-1$.
It remains to ensure that
$\forall f\in F'\cup F\bsp\forall i\in[k',k'')\bsp\exists j\in\{i,i+1\}$
$\sigma''(f\frestr j)=f(j)$ and $\sigma''(f\frestr k'')=f(k'')$.
Let $s\in\Seq_{\leq k''}(b)$ and $|s|>k'$.
If $k'<|s|<k''$, define
$\sigma''(s)=f(|s|)$, if $s\subseteq f\in F'$ and $|s|-k'$ is even
or $s\subseteq f\in F\smallsetminus F'$ and $|s|-k'$ is odd, and
$\sigma''(s)=0$, otherwise.
If $|s|=k''$ define $\sigma''(s)=f(k'')$, if $s\subseteq f\in F'\cup F$, and
$\sigma''(s)=0$, otherwise.
\end{proof}

For $(k,\sigma,T)\in E$, $n=|T|$, and $p\in\Por^b$ let
\[
\textstyle
U_{k,\sigma,T}^p=\set{F\in(\prod b)^n}{p\leq_a(k,\sigma,F)\in Q_{k,\sigma,T}}.
\]

\begin{clm}\label{c8}
$U_{k,\sigma,T}^p$ is a~clopen subset of\/ $(\prod b)^n$ for $n=|T|$.
\qed
\end{clm}

\begin{clm}\label{c9}
\startlist
\begin{enumerate}[label=\rm(\alph*)]
\item\label{c9:a} If $A\subseteq\Por^b$ and\/
$\forall q\in Q_{k,\sigma,T}\bsp\exists p\in A\bsp p\leq_aq$,
then there is $A_0\in[A]^{<\omega}$ such that\/
$\forall q\in Q_{k,\sigma,T}\bsp\exists p\in A_0\bsp p\leq_aq$.

\item\label{c9:b} There is $A_0\in[\Por]^{<\omega}$ such that
$Q_{k,\sigma,T}\subseteq A_0^{\not\perp}$.

\item\label{c9:c} If $A\in[\Por^b]^{<\omega}$ and\/
$\forall q\in Q_{k,\sigma,T}\bsp\exists p\in A\bsp p\leq_aq$, then for every dense
set $D\subseteq\Por^b$ there is $B\in[D]^{<\omega}$ such that
$B$~refines~$A$ and
$\forall q\in Q_{k,\sigma,T}\ \exists p\in B\ p\leq_aq$.

\item\label{c9:d} If $A\in[\Por^b]^{<\omega}$ and\/ $Q_{k,\sigma,T}\subseteq A^{\not\perp}$, 
then for every dense set $D\subseteq\Por^b$ there is
$B\in[D]^{<\omega}$ such that $B$~refines~$A$ and
$\forall q\in Q_{k,\sigma,T}\ \exists p\in B\ p\leq_aq$.
Consequently, $Q_{k,\sigma,T}\subseteq B^{\not\perp}$.
\end{enumerate}
\end{clm}

\begin{proof}
\ref{c9:a}: This follows from \autoref{c8} because $(\prod b)^n$ is a~compact space where $n=|T|$.

\ref{c9:b} is a~consequence of~\ref{c9:a} for $A=\Por^b$ and \ref{c9:c} is a~consequence of~\ref{c9:d},
both according to~\autoref{c7}.

\ref{c9:d}: Let $B'=\set{p\in D}{\exists p'\in A\bsp p\leq p'}$.
Let $q\in Q_{k,\sigma,T}$ and $p'\in A$ be arbitrary such that $p'\parallel q$.
By density of~$D$ there is $p\in D$ such that $p\leq p'$ and $p\leq q$ and hence,
$p\in B'$ and $p\leq_aq$.
Therefore
$\forall q\in Q_{k,\sigma,T}\bsp\exists p\in B'\bsp p\leq_aq$.
By~(a) there is a~finite set $B\subseteq B'$ with the same property.
\end{proof}
Continuing the proof of~\autoref{c6}, $\Por^b=\bigcup_{(k,\sigma,T)\in E}Q_{k,\sigma,T}$ and
by~\autoref{c9} \ref{c9:b} and~\ref{c9:d}, $Q_{k,\sigma,T}$ is $\compact$-linked
taking $Q'=\Por^b$ in the definition.
\end{proof}

Likewise~\autoref{c6}, for any relation~$\rel$ we have the following:

\begin{lemma}
$\Por^b_{\rel}$ is $\sigma$-$\compact$-linked whenever $b\in\baire[(\omega\smallsetminus2)]$.
\qed
\end{lemma}

The following forcing notion increases $\efrak^\forall_{b,\neq}$

\begin{definition}
Let $b$  be as in~\autoref{def:pred}.
\begin{enumerate}
\item The forcing $\Por^{b,\forall}_{\neq}$ is the set of triples
$(k,\sigma,F)$ where $k\in\omega$,
$\sigma\colon\Seq_{<k}(b)\to\omega$ is a~finite partial function mapping
$\Seq_n(b)$ into $b(n)$ for every $n<k$,
$F\subseteq\prod b$ is finite,
\begin{itemize}
\item $f\frestr k\neq g\frestr k$ for all $f\neq g$ belonging to~$F$, and
\item $\set{f\frestr n}{f\in F\setand n<k}\subseteq\dom(\sigma)$;
we require $\dom(\sigma)=\Seq_{<k}(b)$, if $b\in\baire$.
\end{itemize}
\item $\Por^{b,\forall}_{\neq}$ is ordered by
\begin{align*}
(k',\sigma',F')\leq(k,\sigma,F)
{}\Leftrightarrow{}
k'\geq k\comma\sigma'\supseteq\sigma\comma& F'\supseteq F,\setand\\
&\forall f\in F\bsp\forall j\in[k,k')\
f(j)\neq\sigma'(f\frestr j).
\end{align*}
\end{enumerate}   
\end{definition}

Let $G\subseteq\Por^b_{\neq}$ be a $\Por^b_{\neq}$-generic over $V$. Then in $V[G]$, the generic limited predictor is given by
\[\sigma_{\gen,\neq}^b=\bigcup\set{\sigma}{\exists k, F\colon(k,\sigma, F)\in G}.\]

Using a standard density argument, we can get:

\begin{fact}
$\sigma_{\gen,\neq}^b\colon\Seq(b)\to\omega$ is a limited predictor such that
$\forall^\infty j\in\omega$
$x(j)\neq\sigma_{\gen,\neq}^b(x\frestr j)$
for every $x\in V\cap\prod b$.
\end{fact}

\begin{lemma}\label{Pbforalinked}
If $b\in{}^\omega(\omega + 1 \smallsetminus 2)$ and $\lim_{n\to\infty} b(n)=\infty$, then\/
$\Por^{b,\forall}_{\neq}$ is $\sigma$-$n$-linked for every $n\geq1$.
\end{lemma}

\begin{proof}
Let $m_n=\min\set{m\in\omega}{\forall k\geq m\bsp b(k)>n}$.
For every $n\in\omega$ the set
\[
P_n=\set{p\in\Por^{b,\forall}_{\neq}}{k_p\geq m_n}
\]
is an open dense subset of $\Por^{b,\forall}_{\neq}$.
Let $\Sigma_{b,k}=\set{\tau\frestr\Seq_{<k}(b)}{\tau\in\Sigma_b}$.
For $n,k\in\omega$, $\sigma\in\Sigma_{b,k}$, and a~finite set $S\subseteq\Seq_k(b)$
denote
\[
P_{n,k,\sigma,S}=\set{p\in D_n}{k_p=k\comma\sigma_p=\sigma,\set{f\frestr k}{f\in F_p}=S}.
\]
It is easy to see that
$P_n=\bigcup\set{P_{n,k,\sigma,S}}
{k\geq m_n\comma\sigma\in\Sigma_{b,k}\comma S\in[\Seq_k(b)]^{<\omega}}$
is the union of countably many sets.
We prove that every set $P_{n,k,\sigma,S}$ is $n$-linked.

Let $A\in[P_{n,k,\sigma,S}]^n$.
We find a~condition $q\in\Por^{b,\forall}_{\neq}$ that is an extension of conditions in~$A$.
Define $F_q=\bigcup_{p\in A}F_p$ and
$k_q=\min\set{i\geq k}{\forall f,g\in F_q\bsp f\neq g\Rightarrow f\frestr i\neq g\frestr i}$.
We define $\sigma_q\in\Sigma_{b,k_q}$ such that $\sigma_q\supseteq\sigma$ and
$f(j)\neq\sigma_q(f\frestr j)$ for every $f\in F_q$ and $j\in[k,k_q)$.
Let $j\in[k,k_q)$ and $t\in\Seq_j(b)$.
Denote $S_j=\set{f\frestr j}{f\in F_q}$.
If $t\notin S_j$, then let $\sigma_q(t)$ be any value of $b(j)$.
Assume that $t\in S_j$.
For every $p\in A$ there is at most one $f\in F_p$ such that $t\subseteq f$.
Therefore the set $E_t=\set{f\in F_q}{t\subseteq f}$ has cardinality at most~$n$.
Since $j\geq m_n$, $|b(j)|>n$ and therefore it is possible to set $\sigma_q(t)$ to a~value
from $b(j)\smallsetminus\set{f(j)}{f\in E_t}$.
Hence, if $t=f\frestr j$ for some $f\in F_q$ and $j\in[k,k_q)$, then
$f(j)\neq\sigma_q(t)=\sigma_q(f\frestr j)$.
\end{proof}

\subsection*{Forcing \texorpdfstring{$\efrak_2^\const$}{} small}\

The following is based on~\cite{CRS}. Fix $k<\omega$, denote by $\Sigma_2^k$ the collection of maps of $\sigma\colon\bigcup_{i<\omega}2^{ik}\to2^{k}$. Let $\Ebf_2^k=\la\cantor,\Sigma_2^k,{\sqsubset^\star}\ra$ where $x\sqsubset^\star\sigma$ iff $\forall^\infty i\colon x\frestr[ik,(i+1)k)\neq\sigma(x\frestr ik)$. It can be proved that the relational system $\Ebf_2^k$ is a Prs.

Notice that

\begin{fact}[{\cite[Fact.~2.14]{CRS}}]\label{cxn:e_2}
$\Ebf_2^k\leqT\Ebf^\cpr_2$. In particular, $\efrak_2^\const\leq\bfrak(\Ebf_2^k)$ and\/ $\dfrak(\Ebf_2^k)\leq\vfrak_2^\const$.
\end{fact}
\begin{proof}
Note that $\Psi_-\colon\cantor\to\cantor$, defined as the identity map, and $\Psi_+\colon\Sigma_2\to\Sigma_2^k$ is defined as $\Psi_+(\pi)(\sigma)=$ the unique $\tau\in{}^{k}2$
such that $\pi$ predicts $\sigma\char 94\tau$ incorrectly on the whole interval $[ik,(i + 1)k)$ where $|\sigma|=ik$. It is clear that
$(\Psi_-,\Psi_+)$ is the required Tukey connection.
\end{proof}

\begin{lemma}[{\cite{BreShevaIV}, see also~\cite[Lem.~2.15]{CRS}}]\label{c10}
If\/ $\Por$ is $\sigma$-$2^k$-linked, then\/ $\Por$ is\/ $\Ebf_2^k$-good.
\end{lemma}
\begin{proof}
Suppose that $\Por$ is $\sigma$-$2^k$-linked witnessed by $\seq{P_n}{n\in\omega}$. Let $\dot\phi$
is a $\Por$–name for a function $\bigcup_{i<\omega}{}^{ik}2\to{}^{k}2$. For each $n\in\omega$ define $\psi_n\colon\bigcup_{i<\omega}{}^{ik}2\to{}^{k}2$
such that, for each $\sigma\in{}^{ik}2$
\[\text{$\psi_n(\sigma)$ is a $\tau$ such that no $p\in P_n$ forces 
$\dot\phi(\sigma)=\tau$.}\]
Such a $\tau$ clearly exists. Otherwise, for each $\tau\in {}^{k}2$ we could
find $p_\tau\in P_n$ forcing $\dot\phi(\sigma)=\tau$. Since $P_n$ is $2^k$-linked, the $p_\tau$ would have
a common extension which would force $\dot\phi(\sigma)\in {}^{k}2$, a contradiction. Let $H:=\set{\psi_n}{n\in\omega}$.

Now assume that $x\in\cantor$ such that for all $n\in\omega\colon x\not\sqsubset^\star\psi_n$ and show $\Vdash\lqq x\not\sqsubset^\star\dot\phi\rqq$.
Fix $i_0$ and $p\in\Por$. Then there is $n$ such that
$p\in P_n$. We can find $i\geq i_0$ such that $\psi_n(x\frestr ik)=x\frestr[ik,(i + 1)k)$. By
definition of $\psi_n$, there is $q\leq p$ such that $q\Vdash\lqq\dot\phi(x\frestr ik)=\psi_n(x\frestr ik)\rqq$. Thus
$q\Vdash\lqq\dot\phi(x\frestr ik)=x\frestr[ik,(i + 1)k)\rqq$, as required.
\end{proof}

As an immediate consequence, we get

\begin{corollary}\label{c11}
$\sigma$-centered posets are\/ $\Ebf_2^k$-good.
\end{corollary}

\begin{lemma}\label{randomK}
Random forcing is $\sigma$-$\compact$-linked. 
\end{lemma}

\begin{proof}
Represent the random forcing $\Bor$ as the family of Borel subsets of $[0,1]$
of positive Lebesgue measure.
$\Bor=\bigcup_{n\in\omega}Q_n$ where $Q_n=\set{p\in\Bor}{\mu(p)\geq2^{-n}}$.
We prove that for every $n\in\omega$, the set $Q_n$ is $\compact$-linked.
We verify \autoref{c0} taking $Q'=\Bor$.
To see~(a) take any finite partition~$A$ of $[0,1]$ into sets of positive measure;
then $Q_n\subseteq A^{\not\perp}$.

(b)
Assume $A\in[Q']^{<\omega}$ is such that $Q_n\subseteq A^{\not\perp}$ and let
$D\subseteq\Bor$ be an open dense set.
Clearly $\mu(\bigcup A)>1-2^{-n}$.
Find a~finite family $B\subseteq D=D\cap Q'$ refining~$A$ such that $\mu(\bigcup B)>1-2^{-n}$.
Clearly $Q_n\subseteq B^{\not\perp}$.
\end{proof}

In the following, we establish $\Con(\efrak^\const<\efrak^\const_\ubd)$:

\begin{lemma}\label{lemma:b10}
Let $b\in\baire$. Assume\/ $\aleph_1\leq\kappa\leq\lambda=\lambda^{\aleph_0}$ with $\kappa$ regular.
Let $\pi=\lambda\kappa$ and let\/ $\Por^b$ be the FS iteration\/
$\seq{\Por_\alpha,\Qnm_\alpha}{\alpha<\pi}$ where each\/ $\Qnm_\alpha$ is a\/~$\Por_\alpha$-name of~$\Por^b$.
Then $\Por^b$~forces\/ $\cfrak=\lambda$, 
$\baire\eqT\Ebf^\cpr\eqT\Ebf^\cpr_{\leq}\eqT\Cv_{[\lambda]^{<{\aleph_1}}}$, $\kappa\leqT\Cv_\Mwf$, $\Ebf^\cpr_b\leqT\kappa$ for all $b\in\baire$. In particular, it forces\/
$\efrak^\const=\efrak^\const_\leq=\bfrak=\aleph_1\leq\efrak^\const_\ubd=\kappa\leq\vfrak^\const=\vfrak_\leq^\const=\dfrak=\cfrak=\lambda$.   
\end{lemma}

\begin{proof}
Let us define iteration at each $\alpha=\lambda\xi+\varepsilon$ for $\xi<\kappa$ and $\varepsilon<\lambda$ as follows.
For $\xi<\kappa$, 
\begin{enumerate}[label=\rm($\boxplus$)]
\item\label{prof:cl}  an enumeration $\set{b_\varepsilon^\xi}{\varepsilon<\lambda}$ of all the nice $\Por_{\lambda\xi}$-name for all members of $\baire$.
\end{enumerate}
and set $\Qnm_\alpha:=\Por^{b_{\xi}^{\varepsilon}}$ if $\alpha=\lambda\xi+\varepsilon$, $\Por_{\alpha+1}=\Por_\alpha\ast\Qnm_\alpha$, and $\Por_\alpha=\limdir_{\xi<\alpha} \Por_\xi$ if $\alpha$ limit
Notice first that $\Por_\pi$ forces $\cfrak=\lambda$. Since each iterand used in the iteration is $\sigma$-$\compact$-linked by~\autoref{c6}, so by~\autoref{c3}, they are $\Rbf_\leq^\const$-good. So by using~\autoref{Comgood}, $\Por_\pi$ forces $\Cv_{[\pi]^{<\aleph_1}}\leqT \Rbf_\leq^\const$. In particular, $\Por_\pi$ forces $\efrak^\const_{\leq}=\bfrak(\Rbf_\leq^\const)=\aleph_1$ and $\lambda\leq\dfrak(\Rbf_\leq^\const)=\vfrak_\leq^\const$. On the other hand, since $\Vdash_\Por``\cfrak=\lambda"$, $\Por_\pi$ forces $\vfrak^\const=\vfrak_\leq^\const=\cfrak=\lambda$.

In view of~\autoref{Cohenlimit}, $\Por_\pi$ forces $\kappa\leqT\Cv_\Mwf$. We now prove $\Por_\pi$ forces that $ \Ebf^\cpr_{b}\leqT\kappa$ for each $b\in\baire$: We need to define maps $\Psi_-\colon\prod b\to\kappa$ and $\Psi_+\colon\kappa\to\sum_b$ such that, for any $x\in\prod b$ and any $\xi<\kappa$, if $\Psi_-(x)\leq \xi$, then $x\predby^\cpr_{=}\Psi_+(\xi)$. To this end, denote by $\sigma_{\gen,\alpha}$ the $\Ebf^\cpr_{b}$-dominating real over $V_{\alpha}$ added by $\Qnm_{\alpha}$ when $\alpha=\lambda\xi+\varepsilon$ for some $\xi<\kappa$ and $\varepsilon<\lambda$.

By ccc, there exists an $\xi_{b}<\kappa$ such that $b\in V_{\lambda\xi_{b}}$. Next, for $x\in\prod b\cap V_{\lambda\kappa}$, we can find $\xi_b\leq \xi_x<\kappa$ such that $x\in V_{\lambda\xi_x}$, so put $\Psi_-(x):=\xi_x$. 

For $\xi<\kappa$. When $\xi\geq \xi_b$, since $b\in V_{\lambda\xi}$, by~\ref{prof:cl} there is an $\varepsilon<\lambda$ such that $ b=b_{\varepsilon}^\xi$, so define $\Psi_+(\xi):=\sigma_{\gen,\alpha}$ where $\alpha=\lambda\xi+\varepsilon$; otherwise, $\Psi_+(\xi)$ can be anything. It is clear that $(\Psi_-,\Psi_+)$ is the required Tukey connection.
\end{proof}

We now prove $\Con(\efrak_\le^\const<\cov(\Nwf)=\non(\Nwf)<\vfrak_\le^\const)$: 

\begin{lemma}
Assume\/ $\aleph_1\leq\kappa\leq\lambda=\lambda^{\aleph_0}$ with $\kappa$ regular. Let\/ $\Bor_{\pi}$ be a FS iteration of random forcing of length $\pi=\lambda\kappa$. Then, in $V^{\Bor_{\pi}}$, 
$\Ebf_\le^\spr\eqT\Cbf_{[\lambda]^{<{\aleph_1}}}$ and\/ $\Cv_{\Nwf}^\perp\eqT\Mbf\eqT\kappa$.
In particular, it forces\/ 
$\efrak_\le^\const=\aleph_1<\cov(\Nwf)=\non(\Nwf)=\kappa\leq\vfrak_\le^\const=\cfrak=\lambda$. 
\end{lemma}

\begin{proof}
Since each iterands of $\Bor_{\pi}$ is $\Ebf^\cpr_{\leq}$-good because $\Bor$ is $\sigma$-$\compact$-linked by~\autoref{randomK}, by~\autoref{Comgood}, $\Bor_{\pi}$ forces $\Cbf_{[\lambda]^{<{\aleph_1}}}\leqT\Ebf_\le^\spr$.

The rest of the proof is well known, e.g., \cite[Lem.~3.27]{cardonaRb}. 
\end{proof}

Below, we prove $\Con(\efrak^\const_{2}<\efrak_{b,\ne}^\forall=\vfrak_{b,\ne}^\forall<\vfrak^\const_{2})$: 

\begin{lemma}
Let $b\in{}^\omega(\omega\smallsetminus2)$. Assume\/ $\aleph_1\leq\kappa\leq\lambda=\lambda^{\aleph_0}$ with $\kappa$ regular. Let $\pi=\lambda\kappa$ and let $\Por_\pi$ be the FS iteration $\seq{\Por_\alpha,\Qnm_\alpha}{\alpha<\pi}$ where each\/~$\Qnm_\alpha$ is a\/~$\Por_\alpha$-name of\/ $\Por_{\neq}^{b,\forall}$. Then  $\Por_\pi$ forces\/ $\efrak^\const_\leq=\efrak^\const_2=\aleph_1\leq\efrak_{b,\ne}^\forall=\vfrak_{b,\ne}^\forall=\kappa\leq\efrak^\const_2=\cfrak=\lambda$.
\end{lemma}

\begin{proof}
Since each iterands of $\Por_\pi$ is $\Ebf_2^k$-good because $\Por_{\neq}^{b,\forall}$ is $\sigma$-$\infty$-linked by~\autoref{Pbforalinked}, by~\autoref{Comgood}, $\Por_\pi$ forces $\Cbf_{[\lambda]^{<{\aleph_1}}}\leqT\Ebf_2^k$. On the other hand, by~\autoref{thm:fullgen},  $\Por_\pi$~forces $\Ebf^\spr_{b,\neq}\eqT\Mbf\eqT\kappa$. 
\end{proof}

\subsection*{Forcing \texorpdfstring{$\cov(\Nwf)$}{} small}\

Recall:

\begin{definition}[{\cite{KO}}]\label{link}
Let $\rho,\varrho\in\baire$. A forcing notion $\Por$ is \textit{$(\rho,\varrho)$-linked} if there exists a sequence $\seq{Q_{n,j}}{n<\omega,\ j<\rho(n)}$ of subsets of $\Por$ such that
\begin{enumerate}[label= \rm (\roman*)]
\item\label{it:rhopi1} $Q_{n,j}$ is $\varrho(n)$-linked for all $n<\omega$ and $j<\rho(n)$, and
\item\label{it:rhopi2} $\forall\, p\in\Por\ \forall^{\infty}\, n<\omega\ \exists\, j<\rho(n)\colon p\in Q_{n,j}$.
\end{enumerate}
\end{definition}

In~\cite{KO}, Kamo and Osuga propose a gPrs with parameters $\varrho,\rho\in\baire$, denoted as $\aLc^*(\varrho,\rho)$(see~ \cite[Ex.~4.19]{CM}). This work requires just a review of its properties. Assume that $\varrho>0$ and $\rho\geq^* 1$.
.
\begin{enumerate}[label = \rm ($\boxplus_\arabic*$)]
    \item\label{KOb} If $\sum_{i<\omega}\frac{\rho(i)^i}{\varrho(i)}<\infty$, then 
    $\cov(\Nwf)\leq\bfrak(\aLc^*(\varrho,\rho))$ and $\dfrak(\aLc^*(\varrho,\rho))\leq\non(\Nwf)$.
    \item\label{KOc} If $\varrho\not\leq^*1$ and $\rho\geq^*1$, then any $(\rho,\varrho^{\rho^{\id}})$-linked poset is $\aLc^*(\varrho,\rho)$-good (see~\cite[Lem.~10]{KO} and~\cite[Lem.~4.23]{CM}). 
\end{enumerate}

\begin{lemma}\label{Pforall:varrho}
There are increasing functions $b\in{}^\omega(\omega\smallsetminus2)$ and  $\rho,\varrho\in\baire$ such that:
\begin{enumerate}[label=\rm(\arabic*)]
\item $\Por_{b,\ne}^\forall$ is $(\rho, \varrho^{\rho^\id})$-linked; and
\item  $\sum_{n\in\omega} \frac{\rho(n)^n}{\varrho(n)}<\infty$.
\end{enumerate}
\end{lemma}

\begin{proof}
For an increasing finite function $b\in{}^n(\omega\smallsetminus2)$ with $n\geq1$
and for $k\leq n$ let $\Seq_{<k}(b)=\bigcup_{j<k}\prod_{i<j}b(i)$.
Let $\Sigma^n_{b,k}$ denote the set of all predictors $\sigma:\Seq_{<k}(b)\to b(n-1)$
such that $\sigma(s)\in b(|s|)$ whenever $|s|<|n-1|$.
For $k\leq n$, $\sigma\in\Sigma^n_{b,k}$, $\Phi\subseteq\Seq_n(b)=\prod_{k<n}b(k)$
let
\begin{align*}
Q_{n,k,\sigma,\Phi}=\{(k,\sigma,F)\setsep{}&F\subseteq\baire\setand
\set{f\frestr n}{f\in F}=\Phi\\*
&\text{and }\forall f,g\in F\bsp f\neq g\Rightarrow f\frestr k\neq g\frestr k\}.
\end{align*}
Since we do not know the values $b(i)$ for $i\geq n$ in forward we must allow in sets $F$ also
functions that will not apear in conditions of the resulting forcing.
Note that for every increasing function $b'\in\baire$ such that $b\subseteq b'$
every set $Q_{n,k,\sigma,\Phi}\cap\Por_{b',\ne}^\forall$ is $(b'(n)-1)$-linked
(see the proof of \autoref{Pbforalinked}).
The number of triples $(k,\sigma,\Phi)$ is bounded by $N_b=\sum_{k<n}|\Sigma^n_{b,k}|\cdot2^{|\Seq_n(b)|}$.

Now we define the functions $b$, $\rho$, $\varrho$ by induction on $n\in\omega$.
Let $b(0)=2$.
Assume $b\frestr(n+1)$ is defined.
Set $\rho(n)=N_{b\frestr(n+1)}$, $\varrho(n)=\rho(n)^n\cdot(n+1)^2$, and
$b(n+1)=\varrho(n)^{\rho(n)^n}+1$.
\end{proof}

$\Con(\efrak_{b,\ne}^\forall>\cov(\Nwf))$: 

\begin{lemma}
Let $\mu\leq\nu$ be uncountable regular cardinals and let $\lambda\geq\nu$ be a cardinal with $\lambda=\lambda^{<\mu}$. Then, there is a $b\in{}^\omega(\omega\smallsetminus2)$ and a ccc poset that forces\/ $\cov(\Nwf)=\mu$, $\efrak_{b,\ne}^\forall=\vfrak_{b,\ne}^\forall=\nu$, and\/ $\non(\Nwf)=\cfrak$.
\end{lemma}

\begin{proof}
By~\autoref{Pforall:varrho}, there are $b\in{}^\omega(\omega\smallsetminus2)$ and  $\rho,\varrho\in\baire$ such that:
\begin{enumerate}[label=\rm(\arabic*)]
\item $\Por_{b,\ne}^\forall$ is $(\rho, \varrho^{\rho^\id})$-linked; and
\item  $\sum_{n\in\omega} \frac{\rho(n)^n}{\varrho(n)}<\infty$.
\end{enumerate}
  
  Construct $\Por$ by a FS iteration of length $\lambda\nu$ of $\Por_{b,\ne}^\forall$, and of all the ccc subposets of random forcing of size ${<}\mu$ with underlying set an ordinal ${<}\mu$ (like in the previous results). Notice that the iterands are $\mu$-$\aLc^*(\varrho,\rho)$-good by~\ref{KOc} and \autoref{b6} (keep in mind $\Por_{b,\ne}^\forall$ is $(\rho, \varrho^{\rho^\id})$-linked). Hence, by applying~\autoref{Comgood}, we obtain $\Por$ forces $\cov(\Nwf)\leq\bfrak(\aLc^*(\varrho,\rho))\leq\mu$ and $\lambda\leq\dfrak(\aLc^*(\varrho,\rho))\leq\non(\Nwf)$ (see~\ref{KOb}). Moreover, $\Por$ forces $\non(\Nwf)=\lambda=\cfrak$ because it forced that $\cfrak\leq\lambda$.

  On the other hand, by~\autoref{thm:fullgen}, we have $\Por_\pi$~forces $\efrak_{b,\ne}^\forall=\vfrak_{b,\ne}^\forall=\nu$. Lastly, it is not to see that $\Por$ forces $\mu\leq\cov(\Nwf)$.
\end{proof}

\section{Problems}\label{sec:openQ}

While we have shown several results of our newly defined cardinal invariants, see~\autoref{cichonext}, numerous questions remain unresolved:

\begin{question}
Are each one of the following statements consistent with $\thzfc$?  
\begin{enumerate}[label=\rm(\arabic*)]

\item\label{b5:b}
$\efrak^\const_{\hyp,\neq}<\non(\Ewf)$. Dually, $\cov(\Ewf)<\vfrak^\const_{\hyp,\neq}$.

\item $\efrak^\const_\ubd<\efrak^\const_2$. Dually, $\vfrak^\const_2<\vfrak^\const_\ubd$.

\item $\efrak^\const_{\hyp,\neq}<\bfrak$. Dually, $\vfrak^\const_{\hyp,\neq}>\dfrak$. 

\item $\efrak_{b}^\const>\cov(\Nwf)$. Dually, $\vfrak_{b}^\const<\non(\Nwf)$.
\end{enumerate}
\end{question}


{\small
\bibliography{refe}
\bibliographystyle{alpha}
}

\end{document}